\documentclass[reqno]{amsart}
\usepackage[utf8x]{inputenc}  

\usepackage{setspace}
\usepackage[left=3.2cm, right=3.2cm, bottom=4cm]{geometry}                 

\usepackage{graphicx} 
\usepackage{pgf,tikz}
\usetikzlibrary{arrows}
\usepackage{amssymb}
\usepackage{pdfsync}
\usepackage{mathrsfs}
\usepackage{hyperref} 
\usepackage{verbatim} 
\usepackage{epstopdf}
\usepackage{bbm} 
\usepackage[colorinlistoftodos,prependcaption,textsize=tiny]{todonotes}
\usepackage{xargs}

\def\R{\mathbb{R}}

\def\N{\mathbb{N}}
\def\Z{\mathbb{Z}}
\def\Co{\mathbb{C}}
\def\H{\mathbb{H}}

\def\A{\mathcal{A}}

\def\C{\mathcal{C}}

\def\D{\mathcal{D}}
\def\L{\mathcal{L}}

\newcommandx{\emanuel}[2][1=]{\todo[linecolor=green,backgroundcolor=green!25,bordercolor=black,#1]{#2}}

\newcommandx{\diogo}[2][1=]{\todo[linecolor=orange,backgroundcolor=orange!25,bordercolor=orange,#1]{#2}}

\newcommandx{\mateus}[2][1=]{\todo[linecolor=blue,backgroundcolor=blue!25,bordercolor=blue,#1]{#2}}

\newcommandx{\danger}[2][1=]{\todo[linecolor=red,backgroundcolor=red!25,bordercolor=blue,#1]{#2}}

\renewcommand{\d}{\text{\rm d}}
\newcommand{\one}{\mathbbm 1}
\newcommand{\eps}{\varepsilon}

\newcommand{\pvector}[1]{
  \begin{pmatrix}
    #1
  \end{pmatrix}} %
\newcommand{\ddirac}[1]{
  \,\boldsymbol{\delta}\!\pvector{#1}\!} %

\newtheorem{theorem}{Theorem}
\newtheorem{corollary}[theorem]{Corollary}

\newtheorem{proposition}[theorem]{Proposition}
\newtheorem{lemma}[theorem]{Lemma}

\numberwithin{equation}{section}

\title{Extremizers for Fourier restriction on hyperboloids}

\author[Carneiro]{Emanuel Carneiro}
\author[Oliveira e Silva]{Diogo Oliveira e Silva}
\author[Sousa]{Mateus Sousa}

\address{
IMPA - Instituto de Matem\'{a}tica Pura e Aplicada\\
Rio de Janeiro - RJ, Brazil, 22460-320.}
\email{carneiro@impa.br}
\email{mateuscs@impa.br}
\address{
        Hausdorff Center for Mathematics\\
        53115 Bonn, Germany}
\email{dosilva@math.uni-bonn.de}

\date{\today}                                           
\begin{document}

\begin{abstract} The $L^2 \to L^p$ adjoint Fourier restriction inequality on the $d$-dimensional hyperboloid $\H^d \subset \R^{d+1}$ holds provided $6 \leq p < \infty$, if $d=1$, and $2(d+2)/d \leq p\leq  2(d+1)/(d-1)$, if $d\geq2$. Quilodr\'{a}n \cite{Qu15} recently found the values of the optimal constants in the endpoint cases $(d,p)\in\{(2,4),(2,6),(3,4)\}$ and showed that the inequality does not have extremizers in these cases. In this paper we answer two questions posed in \cite{Qu15}, namely: (i) we find the explicit value of the optimal constant in the endpoint case $(d,p) = (1,6)$ (the remaining endpoint for which $p$ is an even integer) and show that there are no extremizers in this case; and (ii) we establish the existence of extremizers in all non-endpoint cases in dimensions $d \in \{1,2\}$. This completes the qualitative description of this problem in low dimensions.
\end{abstract}

\maketitle

\section{Introduction}
The connection between Fourier restriction estimates on smooth hypersurfaces and  Strichartz estimates for linear partial differential equations has been understood for a while.
For instance, Strichartz inequalities for the Schr\"odinger and wave equations correspond to Fourier restriction estimates on the paraboloid and the cone, respectively. 
These are not compact manifolds, but satisfy a scaling symmetry which makes the usual Tomas--Stein argument work.  While the hyperboloid does not possess such a scaling symmetry, it is in some sense well-approximated by the paraboloid and the cone (see Figure \ref{fig:ConeHypParab}) and it serves as an interesting intermediate case where new phenomena emerge.
In this paper, we explore some of these phenomena in the context of sharp Fourier restriction theory.

\begin{figure}
  \centering
  \includegraphics[height=5cm]{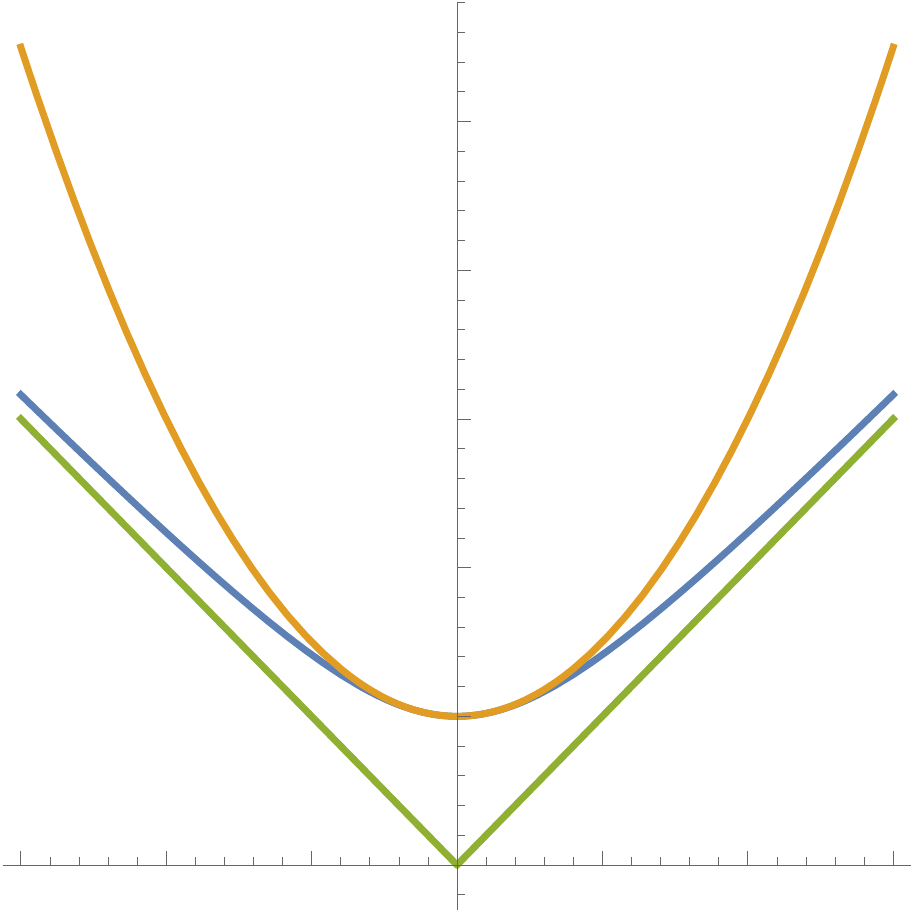} 
    \caption{The paraboloid $y=1+\frac{|x|^2}2$ osculates the hyperboloid $y=\sqrt{1+|x|^2}$ at its vertex. The cone $y=|x|$ approximates the same hyperboloid at infinity.}
\label{fig:ConeHypParab}
\end{figure}

Sharp adjoint Fourier restriction inequalities on the hyperboloid were first studied by Quilodr\'an \cite{Qu15}, who further developed methods from earlier seminal work of Foschi \cite{Fo07} in the context of paraboloids and cones. These works serve as motivation for much of the present paper, and we try to follow the notation and terminology of \cite{Qu15} to facilitate the references.
The hyperboloid $\H^d \subset \R^{d+1}$ is defined by\footnote{A simple rescaling argument transfers all the results of this paper to the hyperboloid $\H^d_s=\big\{(y,y')\in\R^d\times\R: y'=\sqrt{s^2+|y|^2}\big\}$.}
\begin{equation*}
\H^d=\big\{(y,y')\in\R^d\times\R: y'=\sqrt{1+|y|^2}\big\},
\end{equation*}
and comes equipped with the Lorentz invariant measure
\begin{equation}\label{defsigma}
\d\sigma(y,y')=\ddirac{y'-\sqrt{1+|y|^2}}\frac{\d y\,\d y'}{\sqrt{1+|y|^2}},
\end{equation}
which is defined by duality on an appropriate dense class via
$$\int_{\H^d}\varphi(y,y')\,\d\sigma(y,y')=\int_{\R^d} \varphi(y,\sqrt{1+|y|^2})\frac{\d y}{\sqrt{1+|y|^2}}.$$
We normalize the Fourier transform in $\R^{d+1}$ in the following way:
\begin{equation}\label{NormalizeFT}
\widehat{g}(\zeta)=\int_{\R^{d+1}} e^{-iz\cdot\zeta}\, g(z)\,\d z.
\end{equation}
With this normalization, the convolution and the $L^2(\R^{d+1})$-norm satisfy
$$\widehat{g\ast h}=\widehat{g}\cdot\widehat{h}\,; \ \  \text{ and } 
\ \ \|\widehat{g}\|_{L^2(\R^{d+1})}=(2\pi)^{\frac {d+1}2}\|g\|_{L^2(\R^{d+1})}.$$
The {\it Fourier restriction operator} on the  hyperboloid $\H^d$ maps a function $g$ on the ambient space $\R^{d+1}$ to the restriction of its Fourier transform to $\H^d$.
The {\it Fourier extension operator} on $\H^d$ is the adjoint of the Fourier restriction operator, and is given by
$$Tf(x,t)=\int_{\R^d} e^{ix\cdot y} e^{it\sqrt{1+|y|^2}} f(y)\frac{\d y}{\sqrt{1+|y|^2}},$$
where $(x,t)\in\R^d\times\R$ and $f$ belongs to the Schwartz class in $\R^d$.
Here we are identifying a function $f:\H^d\to\Co$ with a complex-valued function on $\R^d$. Its norm in $L^p(\H^d) = L^p(\H^d, \sigma)$ is
$$\|f\|_{L^p(\H^d)}=\left(\int_{\R^d} |f(y)|^p \frac{\d y}{\sqrt{1+|y|^2}}\right)^{\frac 1p}.$$
With the normalization \eqref{NormalizeFT} observe that
\begin{equation}\label{TintermsofHat}
Tf(x,t)=\widehat{f\sigma}(-x,-t).
\end{equation}

The classical work of Strichartz \cite{St77} establishes that
\begin{equation}\label{ExtensionInequality}
\|Tf\|_{L^p(\R^{d+1})}\leq {\bf H}_{d,p} \, \|f\|_{L^2(\H^d)}\,,
\end{equation}
with a finite constant ${\bf H}_{d,p}$ (independent of $f$), provided that
  \begin{equation}\label{AdmissibleRange}
   \begin{cases}
    6 \leq p< \infty, \text{ if } d=1;\\
  \frac{2(d+2)}{d} \leq p\leq  \frac{2(d+1)}{d-1}, \text{ if } d\geq 2.
      \end{cases} 
\end{equation}
For a fixed dimension $d\geq 1$, the lower and upper bounds in the admissible range of exponents $p$ given by \eqref{AdmissibleRange} correspond to the unique exponents for which the extension operator is bounded on the paraboloid and the cone, respectively, each equipped with the appropriate measure (projection measure on the paraboloid and Lorentz invariant measure on the cone).

In this paper we investigate sharp instances of the extension inequality on the hyperboloid.
More precisely, given a pair $(d,p)$ in the admissible range \eqref{AdmissibleRange}, we study {\it extremizers} and {\it extremizing sequences} for inequality \eqref{ExtensionInequality}, and are interested in the value of the {\it optimal constant}
$${\bf H}_{d,p}:=\sup_{0\neq f\in L^2(\H^d)} \frac{\|Tf\|_{L^p(\R^{d+1})}}{\|f\|_{L^2(\H^d)}}.$$
Quilodr\'an \cite{Qu15} studied the endpoint cases $(d,p)\in\{(2,4),(2,6),(3,4)\}$. 
More precisely, he computed the values 
$${\bf H}_{2,4}=2^{\frac 34}\pi, \;{\bf H}_{2,6}=(2\pi)^{\frac 56}, \text{ and } {\bf H}_{3,4}=(2\pi)^{\frac 54},$$
and established the nonexistence of extremizers for the inequality \eqref{ExtensionInequality} associated to these three cases, which are the only ones for which $d>1$ and $p$ is an even integer. The arguments in \cite{Qu15} rely on explicit computations of the $n$-fold convolution of the measure $\sigma$ with itself, and these are computationally challenging if $n\geq 3$ and $d \neq 2$.

Here we answer two questions raised by Quilodr\'an \cite[p. 39]{Qu15}, regarding: (i) the value of the sharp constant and existence of extremizers in the endpoint case $(d,p) = (1,6)$; (ii) the existence of extremizers in the non-endpoint cases in dimensions $d \in \{1, 2\}$. Our results below, together with the previous results of Quilodr\'{a}n \cite{Qu15}, provide a complete qualitative description of this problem in low dimensions.
\begin{theorem}\label{Thm1}
The value of the optimal constant in the case $(d,p)=(1,6)$ is 
$${\bf H}_{1,6}=3^{-\frac1{12}}(2\pi)^{\frac12}.$$ 
Moreover, extremizers for inequality \eqref{ExtensionInequality} do not exist in this case.
\end{theorem}
From Plancherel's Theorem it follows that
$$\|Tf\|_{L^6(\R^2)}^3
=\|(\widehat{f\sigma})^3\|_{L^2(\R^2)}
=\|(f\sigma\ast f\sigma\ast f\sigma)^{\widehat{}}\;\|_{L^2(\R^2)}
=2\pi\|f\sigma\ast f\sigma\ast f\sigma\|_{L^2(\R^2)},$$
which in particular implies that
\begin{equation}\label{BestConstant1D}
{\bf H}^3_{1,6}=2\pi\sup_{0\neq f\in L^2(\H^1)}\frac{\|f\sigma\ast f\sigma\ast f\sigma\|_{L^2(\R^2)}}{\|f\|_{L^2(\H^1)}^3}.
\end{equation}
We are thus led to studying convolution measure $\sigma\ast\sigma\ast\sigma$, a task which we will undertake in greater generality in \S \ref{sec:Convolutions} below. The rigidity of the endpoint lies at the heart of the mechanism responsible for the lack of compactness in these situations (with $p$ even). It would be interesting to investigate if, in all the other endpoint cases $(d,p)$ (now with $p$ not an even integer), one still has lack of extremizers for \eqref{ExtensionInequality}. 

On the other hand, recent works of Fanelli, Vega and Visciglia \cite{FVV11, FVV12} indicate that concentration-compactness arguments may ensure the existence of extremizers in non-endpoint cases for certain families of restriction/extension estimates.  It is important to remark that the problem considered here does not fall under the scope of the methods of \cite{FVV11, FVV12}, since the hyperboloid is a non-compact surface which lacks dilation homogeneity (although many ideas from  \cite{FVV11, FVV12} shall be useful). Our next result establishes that extremizers do exist in {\it every non-endpoint case} of the one- and two-dimensional settings.
\begin{theorem}\label{Thm2}
Extremizers for inequality \eqref{ExtensionInequality} do exist in the following cases:
\begin{itemize}
\item[(a)] $d=1$ and $6<p<\infty$.
\item[(b)] $d=2$ and $4<p<6$.
\end{itemize}
\end{theorem}
As suggested, our proof of Theorem \ref{Thm2} relies on concentration-compactness arguments. The heart of the matter lies in the construction of a {\it special cap}, i.e. a cap that contains a positive universal proportion of the total mass in an extremizing sequence, possibly after applying the symmetries of the problem. This rules out the possibility of ``mass concentration at infinity" and is the missing part in \cite[Proposition 2.1]{Qu15}, which originally outlined the proof of a dichotomy statement for extremizing sequences. The successful quest for a special cap, carried out in \S \ref{sec:SpecialCap}, relies partly on the fact that the lower endpoint $p$ is an even integer in these dimensions, and that the corresponding $(p/2)$-fold convolution of the measure $\sigma$ with itself, when properly parametrized, decays to zero at infinity. In this regard, our argument does not generalize to dimensions $d\geq 3$. Other tools (e.g. coming from bilinear restriction theory, as in \cite{Ca17, FLS16, Ra12}) may be required to address the existence of extremizers in this general non-endpoint setting.
In order to present elementary and self-contained arguments that exploit the convolution structure of the problem, we focus  in this paper  on the lower dimensional cases $d\in\{1,2\}$.
We plan to address the higher dimensional situation in a later work.
 
\subsection{Klein--Gordon propagator} 
 As already pointed out, estimates for Fourier extension operators are related to estimates for dispersive partial differential equations. In our case, the operator $T$ is related to the following Klein--Gordon equation 
\begin{align}\label{kleingordon}
\begin{split}
\partial_t^2u = \Delta_x u&-u,~~(x,t)\in\R^d\times\R ;\\
u(x,0)=u_0(x),&~~~ \partial_t u(x,0)=u_1(x).
\end{split}
\end{align} 
The connection comes from the following  operator, the \emph{Klein--Gordon propagator}, 
\begin{equation*}
    e^{it\sqrt{1-\Delta}}g(x):= \frac{1}{(2\pi)^d}\int_{\R^d}\widehat{g}(\xi)\,e^{i x\cdot\xi}\,e^{it\sqrt{1+|\xi|^2}}\d\xi.
\end{equation*}
Indeed, one can see that solutions to \eqref{kleingordon} can be written as
\begin{align*}
    \begin{split}
    u(\cdot,t)& =\frac{1}{2}\left(e^{it\sqrt{1-\Delta}}u_0(\cdot)-i e^{it\sqrt{1-\Delta}}(\sqrt{1-\Delta})^{-1}u_1(\cdot)\right)+  \\ &  \ \ \ \ \ \ \ \ \ \frac{1}{2}\left(e^{-it\sqrt{1-\Delta}}u_0(\cdot)+i e^{-it\sqrt{1-\Delta}}(\sqrt{1-\Delta})^{-1}u_1(\cdot)\right),
    \end{split}
\end{align*}
and that 
\begin{equation}\label{July19_12:43am}
Tf(x,t)=(2\pi)^{d} \,e^{it\sqrt{1-\Delta}} g(x),
\end{equation} 
where 
\begin{equation}\label{fromextensiontoKG}
\widehat{g}(y):=\frac{f(y)}{\sqrt{1+|y|^2}}.
\end{equation}
This relation implies that the estimate \eqref{ExtensionInequality} is equivalent to 
\begin{equation*}
   \|e^{it\sqrt{1-\Delta}}g\|_{ L_{x,t}^p(\R^d\times\R)}\leq (2\pi)^{-d} \,{\bf H}_{d,p}\,\|g\|_{H^{\frac{1}{2}}(\R^d)},
\end{equation*}
where $H^s(\R^d)$, for $s \geq 0$, is the nonhomogeneous Sobolev space defined by
$$H^{s}(\R^d)=\{g\in L^2(\R^d):~\|g\|_{H^s(\R^d)}^2:=\int_{\R^d}|\widehat{g}(\xi)|^2(1+|\xi|^2)^s\d\xi<\infty\}.$$
This equivalent formulation will be very useful in this paper. In our
concentration-compactness arguments, we explore the fact that convergence of an extremizing sequence $\{f_n\}$ in $L^2(\H^d)$ is equivalent to convergence in $H^{\frac{1}{2}}(\R^d)$ of the sequence $\{g_n\}$ determined by \eqref{fromextensiontoKG}, and, once on the Sobolev side, we may use local compact embeddings. Observe also that, for each $t\in\R$, the operator $e^{it\sqrt{1-\Delta}}$ is unitary in $H^{\frac{1}{2}}(\R^d)$.

\subsection*{Historical remarks} Our results complement the recent, vast and very interesting body of work concerning sharp Fourier restriction and Strichartz estimates.  Sharp Fourier restriction theory has a relatively short history, with the first works on the subject going back to Kunze \cite{Ku03}, Foschi \cite{Fo07} and Hundertmark--Zharnitsky \cite{HZ06}. These works concern extremizers and optimal constants for the Strichartz inequality for the homogeneous Schr\"odinger equation in the lower dimensional cases. These are the cases for which the Strichartz exponent is an even integer, and one can rewrite the left-hand side of the Strichartz inequality  as an ${L}^2$-norm, and invoke Plancherel's Theorem in order to reduce the problem to a multilinear convolution estimate. This subject is becoming increasingly more popular, as shown by the large body of work that appeared in the last decade, and in particular in the last few years.
We mention a few interesting works that deal with sharp Fourier restriction theory on 
spheres \cite{CFOST15, COS15, CS12a, CS12b, Fo15, FLS16, Sh16},
paraboloids \cite{BBCH09, Ca09, CQ14, Go17, Sh09a}, and
cones \cite{BR13, Bu10, Qu13, Ra12}. 
Perturbations of these manifolds have been considered in \cite{FVV11, JSS14, OS14, OS12, OSQ16}.
Sharp bilinear Fourier restriction theory is the subject of \cite{BBJP14, BJO16, J14, OR14}, whereas
other instances of sharp Strichartz inequalities \cite{BBI15}, sharp Sobolev--Strichartz inequalities \cite{FVV12} and sharp Airy--Strichartz inequalities \cite{HS12, Sh09b} have been considered as well.
Finally, we mention a recent survey \cite{FOS17} on sharp Fourier restriction theory which may be consulted for information complementary to that on this Introduction, including a discussion on delta calculus, and further references.

\subsection*{Structure of the paper}
The paper is organized as follows.
In \S \ref{sec:Lorentz} we discuss Lorentz transformations and their relevance to the problem. 
In particular, we decompose the hyperboloid as a disjoint union of {\it caps}, and study how these interact with certain Lorentz transformations.
In \S \ref{sec:Convolutions} we study properties of the $n$-fold convolution of the measure $\sigma$ with itself, explicitly computing some particular instances.
In \S \ref{sec:Nonexistence} we prove Theorem \ref{Thm1}.
The first step is to exhibit an explicit extremizing sequence. 
Once this is done, we appeal to geometric properties of the convolution measure to guarantee that extremizers do not exist.
Finally, \S \ref{sec:SpecialCap} and \S \ref{sec:CC} are devoted to the proof of Theorem \ref{Thm2}. In \S \ref{sec:SpecialCap} we proceed with a detailed construction of a {\it special cap} which contains a non-negligible universal amount of the total mass in an extremizing sequence (properly symmetrized). Once a special cap is available, in \S \ref{sec:CC} we feed this information into the concentration-compactness machinery of Fanelli--Vega--Visciglia \cite{FVV11, FVV12} to ensure that extremizers exist. It is interesting to note that this latter part of the argument works in all dimensions. 

\subsection*{A word on forthcoming notation}
If $x,y$ are real numbers, we will write $x=O(y)$ or $x\lesssim y$ if there exists a finite constant $C$ such that $|x|\leq C|y|$, and $x\simeq y$ if $C^{-1}|y|\leq |x|\leq C|y|$ for some finite constant $C\neq 0$. If we want to make explicit the dependence of the constant $C$  on some parameter $\alpha$, we will  write $x=O_\alpha(y)$ or $x\lesssim_\alpha y$. As is customary, the constant $C$ is allowed to change from line to line. The set of natural numbers is $\N:=\{1,2,3,\ldots\}$. Real and imaginary parts of a complex number $z\in\Co$ will be denoted by $\Re(z)$ and $\Im(z)$, respectively.
The usual inner product between vectors $x,y\in\R^d$ will continue to be denoted by $ x\cdot y$, and we define $\langle x\rangle:=\sqrt{1+|x|^2}$. Given a finite set $A$, we will denote its cardinality by $|A|$.
Finally, $\one_E$ will stand for the indicator function of a given set $E$.

\section{Lorentz invariance}\label{sec:Lorentz}
The measure $\sigma$ defined in \eqref{defsigma} has been referred to as the {\it Lorentz invariant measure} on the hyperboloid. 
This section is meant to explain and expand on this terminology. 
The Lorentz group, denoted by $\L$, is defined as the group of invertible linear transformations in $\R^{d+1}$ that preserve the bilinear form
$$B(x,y)=x_{d+1}y_{d+1}-x_dy_d-\ldots-x_1y_1.$$
In particular, if $L \in \L$, we have $|\det L| = 1$. We denote the subgroup of $\L$ that preserves the hyperboloid $\H^d$ by $\L^+$.
The measure $\sigma$ is likewise preserved under the action of $\L^+$, in the sense that
$$\int_{\H^d} f\circ L \;\d\sigma=\int_{\H^d} f\;\d\sigma,$$
for every $f\in L^1(\H^d)$ and $L\in\L^+$.
This can be readily seen by writing
$$\d\sigma(y,y')=2 \ddirac{y'^2-|y|^2-1}\one_{\{y'>0\}}(y,y')\,\d y\,\d y'.$$
Now, given $t\in(-1,1)$, define the linear map $L^t:\R^{d+1}\to\R^{d+1}$ via
\begin{equation*}
L^t(\xi_1,\ldots,\xi_d,\tau)=\Big(\frac{\xi_1+t\tau}{\sqrt{1-t^2}},\xi_2,\ldots,\xi_d,\frac{\tau+t\xi_1}{\sqrt{1-t^2}}\Big).
\end{equation*}
The family $\{L^t\}_{t\in(-1,1)}$ defines a one-parameter subgroup of $\L^+$.
In particular, the inverse of $L^t$ is $L^{-t}$.
Further notice that, given an orthogonal matrix $A\in\textup{O}(d)$, the transformation $(\xi,\tau)\mapsto \rho_A(\xi,\tau)=(A\xi,\tau)$ belongs to $\L^+$.

As already observed in \cite[\S 3]{Qu15}, given $(\xi,\tau)\in\R^{d+1}$ satisfying $\tau>|\xi|$, a suitable composition of transformations of the form $L^t$ and $\rho_A$ as defined above produces a map $L\in\L^+$, such that 
$$L(\xi,\tau)=(0,\sqrt{\tau^2-|\xi|^2}).$$
This observation will simplify several computations involving convolutions of the measure $\sigma$ with itself, which we explore in the next section.
Given $p\in[1,\infty]$, $L\in\L^+$ and $f\in L^p(\H^d)$, define the composition $L^* f=f\circ L$. The considerations made so far imply that 
\begin{equation}\label{NormInvariance}
\|L^* f\|_{L^p(\H^d)}=\|f\|_{L^p(\H^d)},\text{ and }\|T(L^*f)\|_{L^p(\R^{d+1})}=\|Tf\|_{L^p(\R^{d+1})}.
\end{equation}
In particular, if $\{f_n\}_{n\in\N}$ is an extremizing sequence for inequality \eqref{ExtensionInequality} and $\{L_n\}_{n\in\N}\subset\L^+$, then $\{L_n^* f_n\}_{n\in\N}$ is still an extremizing sequence for inequality \eqref{ExtensionInequality}.

The Lorentz invariance just discussed will be crucial in several of our arguments, as it allows to localize the action to a fixed bounded region. We now detail this principle in the lower dimensional setting $d\in\{1,2\}$.

\subsection{One-dimensional tessellations}
Let us define a {\it one-dimensional cap} to be a set of the form
\begin{align}\label{July21_7:34pm}
\begin{split}
\C_k&:=\{(\xi,\tau)\in\H^1:\sinh(k-1/2)\leq\xi<\sinh(k+1/2)\}    \\
    &\,\,=\{(\sinh(u),\cosh(u)) \in \R^2:k-1/2\leq u<k+1/2\}\,,
\end{split}
\end{align}
for some $k\in\Z$. The following simple result already illustrates the main point.

\begin{lemma}\label{1Dtessellation}
Let $k\in\Z$, and let $\C_k\subset\H^1$ be the corresponding one-dimensional cap.
Then:
\begin{itemize}
\item[(a)] $\sigma(\C_k)=1$.
\item[(b)] There exists $t_k\in(-1,1)$, such that
$L^{t_k}(\C_k)=\C_0.$
\end{itemize}
\end{lemma}

\begin{proof}
The proof of part (a) amounts to a straightforward change of variables:
$$\sigma(\C_k)
=\int_{\sinh(k-\frac12)}^{\sinh(k+\frac12)}\frac{\d y}{\sqrt{1+y^2}}
=\int_{k-\frac12}^{k+\frac12}\frac{\cosh (u)\d u}{\sqrt{1+\sinh (u)^2}}=1.
$$
For part (b), let $t_k=-\tanh(k)\in(-1,1)$. Then the Lorentz transformation $L^{t_k}$ provides a bijection between the caps $\C_k$ and $\C_0$. That $L^{t_k}(\C_k)=\C_0$ follows from
\begin{align*}
L^{-t_k}(\sinh(u),\cosh(u))&=\left(\frac{\sinh(u)+\tfrac{\sinh(k)}{\cosh(k)} \cosh(u)}{\sqrt{1-\tfrac{\sinh^2(k)}{\cosh^2(k)}}},\frac{\cosh(u)+\tfrac{\sinh(k)}{\cosh(k)}\sinh(u)}{\sqrt{1-\tfrac{\sinh^2(k)}{\cosh^2(k)}}}\right) \\
&=(\sinh(u+k),\cosh(u+k)).    
\end{align*}
This concludes the proof of the lemma.
\end{proof}
\begin{figure}
  \centering
   \includegraphics[height=6cm]{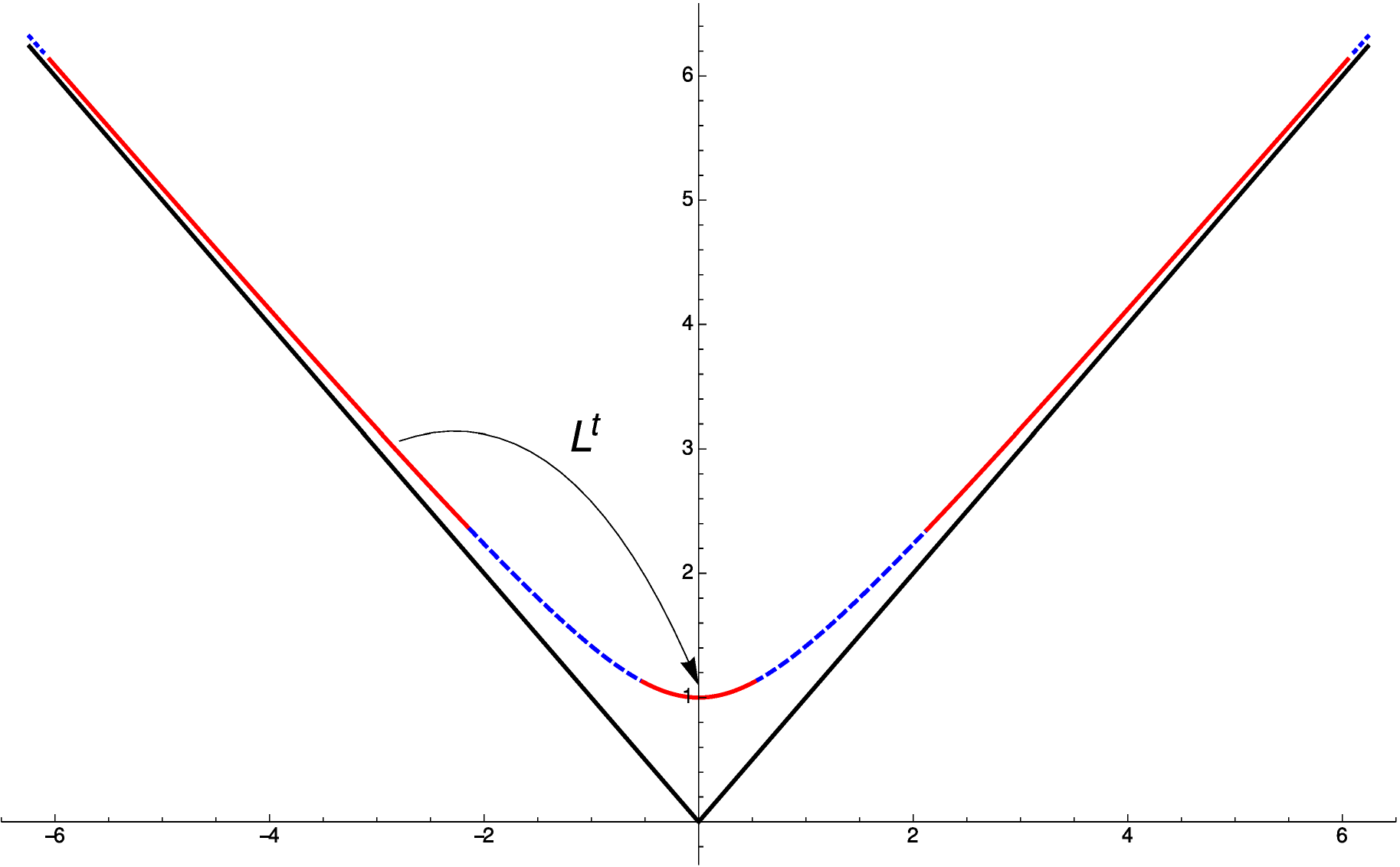}
   \caption{The one-dimensional cap movement: a carefully chosen Lorentz transformation interchanges the caps. Here, $L^{t}(\C_{-2})=\C_0$ for $t = \tanh(2)$.}
\label{fig:1DCapMovement}
\end{figure}

\subsection{Two-dimensional tessellations}
We now define a {\it two-dimensional cap} to be a set of the form
\begin{equation}\label{2DCap}
\C_{n,j}:=\left\{(r\cos\theta,r\sin\theta,\langle r\rangle)\in\H^2:\;2^n\leq r< 2^{n+1}\text{ and }\frac{2\pi j}{2^n}\leq\theta<\frac{2\pi(j+1)}{2^n}\right\},
\end{equation}
for some $n\in\N$ and $0\leq j<2^n$, and additionally we consider
\begin{equation}\label{2DCap0}
\C_{0,0}:=\{(\xi,\tau)\in\H^2: |\xi|<2\}.
\end{equation}
Grouping together the caps of the $n$-th generation, we notice that the hyperboloid $\H^2$ is partitioned into a disjoint union of annuli,
\begin{equation}\label{DefAnnulus}
\H^2=\bigcup_{n=0}^\infty\A_n, \text{ where }\mathcal{A}_n:=\bigcup_{j=0}^{2^n-1} \C_{n,j}.
\end{equation}
See Figure \ref{fig:Ourico} for an illustration of these decompositions.

\begin{figure}
  \centering
  \includegraphics[height=6cm]{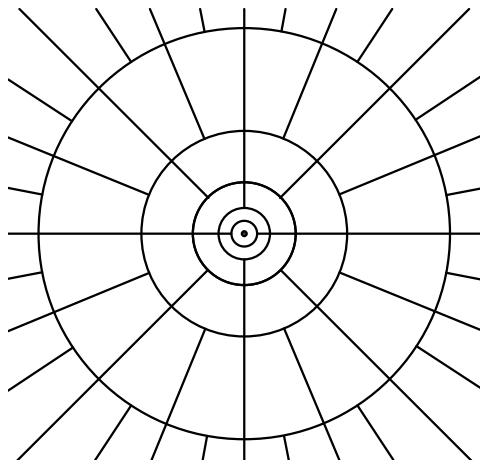}
    \caption{Projection of the tessellation of $\H^2$ into caps $\{\C_{n,j}\}$ onto the horizontal plane $\tau=0$.}
\label{fig:Ourico}
\end{figure}

Given $\varphi\in[0,2\pi)$, denote by $R_\varphi$ the rotation in $\R^3$ by angle $\varphi$ around the vertical $\tau-$axis:
$$R_\varphi(\xi_1,\xi_2,\tau)=(\xi_1\cos\varphi+\xi_2\sin\varphi,-\xi_1\sin\varphi+\xi_2\cos\varphi,\tau).$$
The next result is the two-dimensional equivalent of Lemma \ref{1Dtessellation}, and in particular shows that any cap can be mapped into the ball of radius $2\sqrt{2}\pi$ centered at the origin by an appropriate composition of Lorentz transformations. See Figure \ref{fig:2DCapmovement} for an illustration of these movements.

\begin{lemma}\label{2Dtessellation}
Let $n\in\N_0$ and $0\leq j<2^n$, and let $\C_{n,j}\subset\H^2$ be the corresponding two-dimensional cap. Then: 
\begin{itemize} 
\item[(a)] $\sigma(\C_{n,j})\simeq 1$.
\item[(b)] There exists $t\in [0,1)$ and $\varphi\in [0,2\pi)$, such that 
$$(L^{-t}\circ R_\varphi)(\C_{n,j})
\subseteq 
\{(\xi,\tau)\in\H^2: |\xi|\leq 2\sqrt{2}\pi\}.$$ 
\end{itemize}
\end{lemma}

\begin{proof}
Let $n\in\N$ and $0\leq j<2^n$.
A computation in polar coordinates shows that 
$$\sigma(\C_{n,j})=
\int_{2^n}^{2^{n+1}}\left(\int_{\frac{2\pi j}{2^n}}^{\frac{2\pi (j+1)}{2^n}}\d\theta\right) \frac{r\d r}{\sqrt{1+r^2}}
=\frac{2\pi}{2^n}\left(\sqrt{1+4^{n+1}}-\sqrt{1+4^{n}}\right),
$$
from which one easily checks that 
$$\frac 9{10}\leq\frac{\sigma(\C_{n,j})}{2\pi}\leq 1.$$
Moreover, ${\sigma(\C_{0,0})}=2\pi(\sqrt{5}-1)$, and so one sees that the $\sigma$-measure of any two-dimensional cap is comparable to 1. This establishes part (a).

For part (b), we lose no generality in assuming $n\geq 3$, for otherwise we can simply take $t=\varphi=0$.
Given such $n$, and $0\leq j<2^n$, choose $\varphi\in[0,2\pi)$ so that
$$R_\varphi(\C_{n,j})
\subseteq
\Big\{(r\cos\theta,r\sin\theta,\langle r\rangle)\in\H^2: 2^n\leq r<2^{n+1}\text{ and }|\theta|\leq\frac{\pi}{2^n}\Big\}.$$
Let $t:={1-(\frac{\pi}{2^n})^2}$, which is nonnegative since $n\geq 3$. Noting that
$$(L^{-t}\circ R_\varphi)(\C_{n,j})
\subseteq
\left\{\left(\frac{r\cos\theta-t\langle r\rangle}{\sqrt{1-t^2}},r\sin\theta,\frac{\langle r\rangle-tr\cos\theta}{\sqrt{1-t^2}}\right)\in\H^2: 2^n\leq r<2^{n+1}\text{ and }|\theta|\leq\frac{\pi}{2^n}\right\},$$
it suffices to check that
$$\left|\frac{r\cos\theta-t\langle r\rangle}{\sqrt{1-t^2}}\right|\leq 2\pi,
\text{ and }
|r\sin\theta|\leq 2\pi.$$
Observe that $r\leq\langle r\rangle$ and $\cos\theta\geq \cos(\frac{\pi}{2^n})\geq1-\frac12(\frac{\pi}{2^n})^2>t$. We first claim that $r\cos\theta-t\langle r\rangle \geq 0$. In fact, using the fact that $1 + \frac{x}{2} \geq \sqrt{1 +x}$, we have
\begin{align*}
\frac{\cos \theta}{t} \geq \frac{1-\frac12(\frac{\pi}{2^n})^2}{1- (\frac{\pi}{2^n})^2} \geq 1 + \frac{1}{2^{2n+1}} \geq \sqrt{1 + \frac{1}{2^{2n}}} \geq \sqrt{1 + \frac{1}{r^2} } =  \frac{\langle r\rangle}{r}.
\end{align*}
Therefore it follows that 
$$\left|\frac{r\cos\theta-t\langle r\rangle}{\sqrt{1-t^2}}\right|
\leq\frac{r(\cos\theta-t)}{\sqrt{1-t^2}}
\leq\frac{r(1-t)}{\sqrt{1-t^2}}
=r\sqrt{\frac{1-t}{1+t}}
\leq r\sqrt{1-t}
<2^{n+1}\frac{\pi}{2^n}=2\pi.$$
Noting that $\frac{\sin(x)}{x}\leq 1$, we similarly have that
$$|r\sin\theta|<2^{n+1}\sin\left(\frac{\pi}{2^{n}}\right)=2\pi\frac{\sin(\frac{\pi}{2^{n}})}{\frac{\pi}{2^n}}\leq 2\pi.$$
This concludes the proof of the lemma.
\end{proof}

\begin{figure}
  \centering
 \includegraphics[height=6cm]{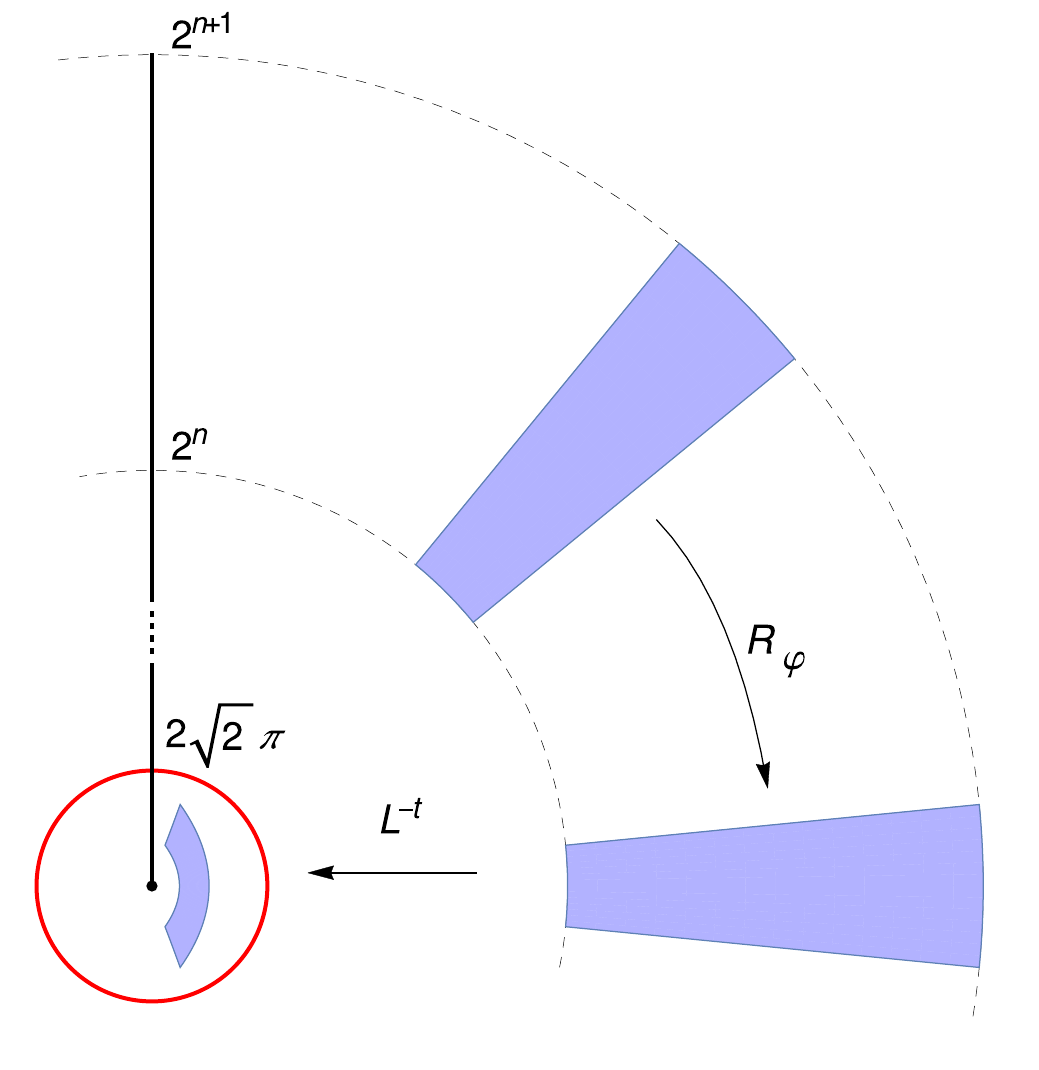}
    \caption{Projection of the two-dimensional cap movement: a rotation followed by a Lorentz transformation moves the cap inside the set $\{(\xi,\tau)\in\H^2: |\xi|\leq 2\sqrt{2}\pi\}$.}
\label{fig:2DCapmovement}
\end{figure}

\section{Convolutions}\label{sec:Convolutions}
In this section, we collect some facts about convolution measures that will be relevant in the sequel.
We start with some general considerations which hold in arbitrary dimensions $d\geq 1$. Let $\sigma^{(\ast n)}=\sigma\ast\ldots\ast\sigma$ denote the $n$-fold convolution of the Lorentz invariant measure $\sigma$ defined in \eqref{defsigma} with itself.
If $n\geq 2$, then the convolution measure $\sigma^{(\ast n)}$ is absolutely continuous with respect to Lebesgue measure on $\R^{d+1}$, and it is supported in the closure of the region
\begin{equation}\label{SuppConvolution}
\mathcal{P}_{d,n}:=\{(\xi,\tau)\in \R^d\times\R:\tau>\sqrt{n^2+|\xi|^2}\}.
\end{equation}
The Lorentz invariance discussed in the previous section implies that $\sigma^{(\ast n)}$ is constant along certain hyperboloids. More precisely, if $(\xi,\tau)\in\mathcal{P}_{d,n}$, then
\begin{equation}\label{LorentzInv}
\sigma^{(\ast n)}(\xi,\tau)=\sigma^{(\ast n)}(0,\sqrt{\tau^2-|\xi|^2}).
\end{equation}
The next result establishes some basic convolution properties on the one-dimensional hyperbola $(\H^1,\sigma)$. 

\begin{lemma}\label{ConvolutionBasicProperties}
Let $\sigma$ denote the Lorentz invariant measure on the hyperbola $\H^1$.
Then, for every $(\xi,\tau)\in \R\times\R$,
\begin{itemize}
\item[(a)] The convolution measure $\sigma\ast\sigma$ is given by 
$$(\sigma\ast\sigma)(\xi,\tau)=\frac{4}{\sqrt{\tau^2-\xi^2}\sqrt{\tau^2-\xi^2-4}}\, \one_{\{\tau\geq \sqrt{2^2+\xi^2}\}}(\xi,\tau).$$
\item[(b)] The following recursive formula holds for $n\geq 2$:
$$
\sigma^{(\ast (n+1))}(\xi,\tau)
=4\int_n^{\sqrt{\tau^2-\xi^2}-1}\frac{x \;\sigma^{(\ast n)}(0, x)}{(\sqrt{\tau^2-\xi^2}+1)^2-x^2)^{\frac12}(\sqrt{\tau^2-\xi^2}-1)^2-x^2)^{\frac 12}}\,\d x.
$$
\end{itemize}
\end{lemma}

\begin{proof}
We start with part (a).
By the Lorentz invariance \eqref{LorentzInv}, it suffices to prove that
\begin{equation}\label{ConvolutionAfterLorentz}
(\sigma\ast\sigma)(0,\tau)=\frac{4}{\tau\sqrt{\tau^2-4}}\one_{\{\tau\geq 2\}}(\tau).
\end{equation}
This can be obtained as follows: first of all,
$$(\sigma\ast\sigma)(0,\tau)
=\int_{\R} \ddirac{\tau-2\langle y\rangle} \frac{\d y}{\langle y\rangle^2}
=2\int_0^\infty \ddirac{\tau-2\langle y\rangle} \frac{\d y}{\langle y\rangle^2}.$$
Changing variables $u=\langle y\rangle$, and then $v=2u$, we have that
$$(\sigma\ast\sigma)(0,\tau)
=2\int_1^\infty \ddirac{\tau-2u}\frac{1}{u^2}\frac{u}{\sqrt{u^2-1}}\,\d u
=4\int_2^\infty  \frac{\ddirac{\tau-v}}{v\sqrt{v^2-4}}\,{\d v}.$$
This implies \eqref{ConvolutionAfterLorentz} at once, and finishes the proof of part (a).

We now turn to the proof of part (b). 
Again by Lorentz invariance, it suffices to establish
\begin{equation}\label{RecursiveFormula}
\sigma^{(\ast (n+1))}(0,\tau)
=4\int_n^{\tau-1}\frac{x\;\sigma^{(\ast n)}(0, x)}{\sqrt{(\tau+1)^2-x^2}\sqrt{(\tau-1)^2-x^2}}\,\d x.
\end{equation}
We proceed by induction on $n$.
Since $\sigma^{(\ast n)}$ is a function by hypothesis, the $(n+1)$-fold convolution can be obtained by convolving that function with the measure $\sigma$, as follows:
\begin{align*}
\sigma^{(\ast (n+1))}(0,\tau)
&=\int_{\H^1}\sigma^{(\ast n)}((0,\tau)-(y,y'))\, \d\sigma(y,y')\\
&=\int_{\R}\sigma^{(\ast n)}(-y,\tau-\langle y\rangle) \frac{\d y}{\langle y\rangle}\\
&=2\int_{0}^\infty\sigma^{(\ast n)}(0,\sqrt{\tau^2-2\tau\langle y\rangle+1}) \frac{\d y}{\langle y\rangle},
\end{align*}
where the Lorentz invariance \eqref{LorentzInv} was again used in the last identity.
Changing variables $u=\langle y\rangle$ as before, we have that:
$$\sigma^{(\ast (n+1))}(0,\tau)
=2\int_{1}^{\frac{\tau^2+1-n^2}{2\tau}} \frac{\sigma^{(\ast n)}(0,\sqrt{\tau^2-2\tau u+1})}{\sqrt{u^2-1}}\,\d u,$$
where the upper limit in the region of integration is due to support considerations involving \eqref{SuppConvolution}. Changing variables $v=\tau^2-2\tau u+1$, we continue to compute:
$$\sigma^{(\ast (n+1))}(0,\tau)
=2\int_{n^2}^{(\tau-1)^2} \frac{\sigma^{(\ast n)}(0,\sqrt{v})}{\sqrt{(\tau^2-v+1)^2-(2\tau)^2}} \,\d v.$$
A final change of variables $x=\sqrt{v}$ yields the desired formula \eqref{RecursiveFormula}. This finishes the proof.
\end{proof}

Identities \eqref{ConvolutionAfterLorentz} and \eqref{RecursiveFormula} for $n=2$ imply the following integral formula for the $3$-fold convolution measure which should be compared to \cite[Lemma 8]{CFOST15}: If $\tau>3$, then
\begin{equation}\label{FormulaSigma3}
\sigma^{(\ast 3)}(0,\tau)
=
16\int_2^{\tau-1}\frac{1}{\sqrt{(\tau+1)^2-x^2}\sqrt{(\tau-1)^2-x^2}}\frac{\d x}{\sqrt{x^2-4}}.
\end{equation}
This integral representation is amenable to a robust numerical treatment with {\it Mathematica}, see Figure \ref{fig:UpperLower} below. 
It is also the starting point for the study of the basic properties of the convolution measure $\sigma^{(\ast 3)}$, which are summarized in the following result.

\begin{lemma}\label{Sigma3Properties}
Let $\sigma$ denote the Lorentz invariant measure on the hyperbola $\H^1$.
Then the function $\tau\mapsto\sigma^{(\ast 3)}(0,\tau)$ is continuous on the half-line $\tau>3$. 
It extends continuously to the boundary of its support, in such a way that
\begin{equation}\label{July13_13:00pm}
\sup_{\tau>3}\sigma^{(\ast 3)}(0,\tau)=\sigma^{(\ast 3)}(0,3)=\frac{2\pi}{\sqrt{3}}\,,
\end{equation}
and this global maximum is strict, i.e.
\begin{equation}\label{July13_11:21am}
\sigma^{(\ast 3)}(0,\tau)<\frac{2\pi}{\sqrt{3}},\text{ for every }\tau>3.
\end{equation}
In particular, this implies that 
\begin{align}\label{July13_12:58pm}
{\bf H}_{1,6}\leq3^{-\frac1{12}}(2\pi)^{\frac12}.
\end{align}
\end{lemma}

\begin{proof}
An application of Lebesgue's Dominated Convergence Theorem to the integral \eqref{FormulaSigma3} establishes that the function $\tau\mapsto\sigma^{(\ast 3)}(0,\tau)$ is continuous for $\tau>3$. We can appeal to the same formula to crudely estimate:
\begin{equation}\label{Sigma3LowerBound}
\sigma^{(\ast 3)}(0,\tau) \geq L(\tau),
\text{ for every } \tau>3,
\end{equation}
where $L$ denotes the lower bound
$$L(\tau):=\frac{16\; I(\tau)}{\sqrt{(\tau+1)^2-2^2}\sqrt{(\tau-1)+(\tau-1)}\sqrt{(\tau-1)+2}},$$
and $I(\tau)$ denotes the integral 
$$I(\tau):=\int_2^{\tau-1}\frac{1}{\sqrt{(\tau-1)-x}}\frac{\d x}{\sqrt{x-2}}.$$
Via the affine change of variables $x\mapsto(\tau-3)x+2$, we see that  $I(\tau)=\pi$, for every $\tau>3$. Substituting in \eqref{Sigma3LowerBound}, we have that
\begin{equation}\label{UpperBoundValueat3}
\liminf_{\tau\to 3^+} \sigma^{(\ast 3)}(0,\tau)\geq \frac{2\pi}{\sqrt{3}}.
\end{equation}
Crude upper bounds of similar flavor yield 
\begin{equation}\label{Sigma3UpperBound}
\sigma^{(\ast 3)}(0,\tau)\leq U(\tau), \text{ for every } \tau>3,
\end{equation}
where the upper bound $U$ is given by
\begin{equation*}
U(\tau):=\frac{16\pi}{\sqrt{(\tau+1)^2-(\tau-1)^2}\,\sqrt{(\tau-1)+2}\,\sqrt{2+2}}.
\end{equation*}
Incidentally, note that this implies $\sigma^{(\ast 3)}(0,\tau)\lesssim \tau^{-1}$, for large values of $\tau$.
It follows from \eqref{Sigma3UpperBound} that
\begin{equation}\label{LowerBoundValueat3}
\limsup_{\tau\to 3^+} \sigma^{(\ast 3)}(0,\tau)\leq  \frac{2\pi}{\sqrt{3}}.
\end{equation}
Estimates \eqref{UpperBoundValueat3} and \eqref{LowerBoundValueat3} together imply 
$$\lim_{\tau\to 3^+} \sigma^{(\ast 3)}(0,\tau)=  \frac{2\pi}{\sqrt{3}}.$$
Noting that the upper bound $U$ satisfies $U(3)=\frac{2\pi}{\sqrt{3}}$, and that 
$$U'(\tau)=-\frac{2\pi (2\tau+1)}{\tau^{\frac 32}(\tau+1)^{\frac 32}}<0\  \text{ for every }\tau>3\,,$$
we arrive at \eqref{July13_11:21am}.

\begin{figure}
  \centering
  \includegraphics[height=5cm]{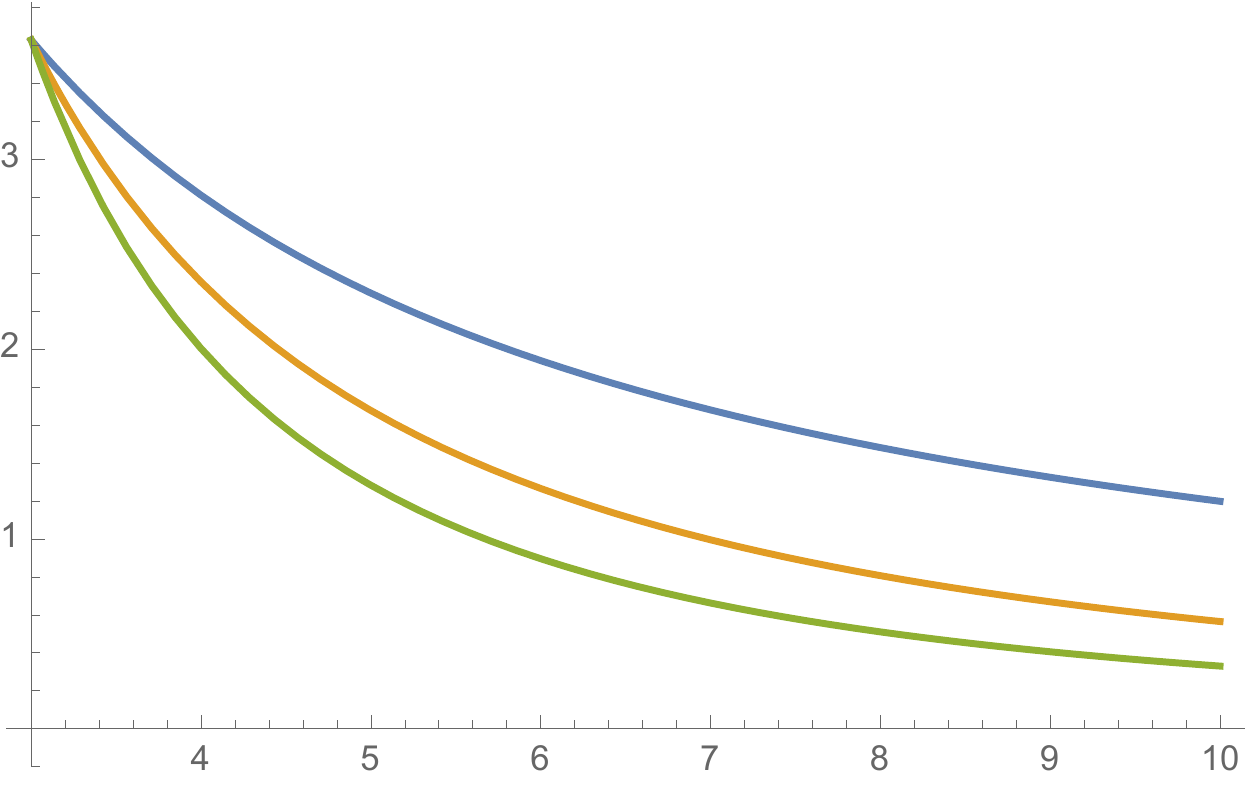} 
    \caption{The lower bound $L(\tau)$ and the upper bound $U(\tau)$ for the function $\tau\mapsto\sigma^{(\ast 3)}(0,\tau)$ on the half-line $\tau>3$.}
\label{fig:UpperLower}
\end{figure}

Finally, letting $\boldsymbol{\delta}_2$ denote the two-dimensional Dirac delta, we have
\begin{align*}
& \|f\sigma\ast f\sigma\ast f\sigma\|_{L^2(\R^2)}^2  = \int_{(\R^2)^6} f(x_1)\,f(x_2)\,f(x_3) \overline{f(x_4)}\,\overline{f(x_5)}\,\overline{f(x_6)} \, \d\Sigma,
\end{align*}
where $\d\Sigma = \d\Sigma(x_1,\ldots, x_6) = \boldsymbol{\delta}_2(x_1+x_2+x_3 - x_4 - x_5-x_6)\, \d\sigma(x_1)\ldots\d\sigma(x_6)$. An application of the Cauchy--Schwarz inequality leads to 
\begin{align}\label{July13_13:01pm}
\begin{split}
& \|f\sigma\ast f\sigma\ast f\sigma\|_{L^2(\R^2)}^2  \leq \int_{(\R^2)^6} |f(x_1)|^2\,|f(x_2)|^2\,|f(x_3)|^2\, \d\Sigma\\
&  \ \ \ \ \ \ \ = \int_{(\R^2)^3} |f(x_1)|^2\,|f(x_2)|^2\,|f(x_3)|^2\,\sigma^{(\ast 3)}(x_1+x_2+x_3) \,\d\sigma(x_1)\,\d\sigma(x_2)\,\d\sigma(x_3)\\
& \ \ \ \ \ \ \ \leq \sup_{x \in \R^2} \sigma^{(\ast 3)}(x) \,\|f\|_{L^2(\H^1)}^6.
\end{split}
\end{align}
Estimate \eqref{July13_12:58pm} now follows from \eqref{BestConstant1D}, \eqref{July13_13:00pm} and \eqref{July13_13:01pm}. This completes the proof of the lemma.
\end{proof}

Two-dimensional counterparts of the results from this section were obtained in \cite[Lemma 5.1]{Qu15}. 
We record them here.

\begin{lemma}[cf. \cite{Qu15}]\label{2DConvolutions}
Let $\sigma$ denote the Lorentz invariant measure on the hyperboloid $\H^2$.
Then, for every $(\xi,\tau)\in\R^2\times\R$,
\begin{itemize}
\item[(a)]$(\sigma\ast\sigma)(\xi,\tau)=\frac{2\pi}{\sqrt{\tau^2-|\xi|^2}}\one_{\{\tau\geq\sqrt{2^2+|\xi|^2}\}}(\xi,\tau),$
\item[(b)]$(\sigma\ast\sigma\ast\sigma)(\xi,\tau)=(2\pi)^2\Big(1-\frac{3}{\sqrt{\tau^2-|\xi|^2}}\Big)\one_{\{\tau\geq\sqrt{3^2+|\xi|^2}\}}(\xi,\tau).$
\end{itemize}
\end{lemma}

\section{Nonexistence of extremizers at the endpoint $(d,p)=(1,6)$}\label{sec:Nonexistence}
This section is devoted to the proof of Theorem \ref{Thm1}. After studying the behavior of $\sigma^{(\ast 3)}$ in \S \ref{sec:Convolutions}, the material in this section is partly motivated by the outline of \cite[Appendices B and C]{Qu15}. The heart of the matter lies in the construction of an explicit extremizing sequence for inequality \eqref{ExtensionInequality}, which is the content of the next result.

\begin{proposition}\label{ExtremizingSequence}
Let $\sigma$ denote the Lorentz invariant measure on the hyperbola $\H^1$.
Given $a>0$, let $f_a(y)=e^{-a\langle y\rangle}$, $y\in\R$. Then:
\begin{itemize}
\item[(a)] For every $n\in\N$ we have 
$$(f_a\sigma)^{(\ast n)}(\xi,\tau)=e^{-a\tau}\sigma^{(\ast n)}(\xi,\tau).$$
\item[(b)]
The function $a\mapsto \sqrt{a} \,e^{2a} \|f_a\|_{L^2(\H^1)}^2$ is bounded on the half-line $a>0$, and satisfies 
\begin{equation}\label{BehaviorAtInfty}
\lim_{a\to\infty}\sqrt{a}\, e^{2a}\|f_a\|_{L^2(\H^1)}^2=\sqrt{\pi}.
\end{equation}
\item[(c)] The sequence $\{f_a\}_{a\in\N}$ satisfies
\begin{equation}\label{PreBestConstant}
\lim_{a\to\infty}\frac{\|f_a\sigma\ast f_a\sigma\ast f_a\sigma\|_{L^2(\R^2)}^2}{\|f_a\|_{L^2(\H^1)}^6}
=\frac{2\pi}{\sqrt{3}},
\end{equation}
and is extremizing for inequality \eqref{ExtensionInequality} when $(d,p)=(1,6)$, as 
$a\to\infty$. In particular,
\begin{align*}
{\bf H}_{1,6} = 3^{-\frac1{12}}(2\pi)^{\frac12}.
\end{align*}
\end{itemize}
\end{proposition}

\begin{proof}
The proof of (a) is analogous to part of the proof of \cite[Lemma B.1]{Qu15}. 
We present the details for the convenience of the reader.
Letting $g_a(\xi,\tau)=e^{-a\tau}$, we have that $(f_a\sigma)^{(\ast n)}=(g_a\sigma)^{(\ast n)}$.
Therefore,
$$(f_a\sigma)^{(\ast n)}(\xi,\tau)=(g_a\sigma)^{(\ast n)}(\xi,\tau)=g_a(\xi,\tau) \sigma^{(\ast n)}(\xi,\tau)=e^{-a\tau} \sigma^{(\ast n)}(\xi,\tau),$$
where the second identity follows from the fact that $g_a$ is the exponential of a linear function.

For part (b), change variables $\langle y\rangle=\cosh t$ to compute
\begin{equation}\label{HrelK}
\|f_a\|_{L^2(\H^1)}^2=\int_\R e^{-2a\langle y\rangle} \frac{\d y}{\langle y\rangle}=2\int_0^\infty e^{-2a\cosh t} \d t=2K_0(2a).
\end{equation}
Here, the modified Bessel function of the second kind $K_\nu$ is given for $\Re(z)>0$ by
\begin{equation*}
K_\nu(z)=\int_0^\infty \exp(-z\cosh t)\cosh(\nu t) \,\d t.
\end{equation*}
Claim \eqref{BehaviorAtInfty} boils down to the well-known fact
\begin{equation}\label{KBesselAtInfty}
\lim_{a\to\infty}\sqrt{a} \,e^{2a} K_0(2a)=\frac{\sqrt{\pi}}{2},
\end{equation}
see e.g. \cite[\S 7.34 (1)]{Wa66}.
We finish the proof of part (b) by invoking the facts  that $K_0(x)\lesssim |\log(x)|$ as $x\to 0^+$ (see e.g. \cite[formula (9.6.8) on p. 375]{AS70}), and that $K_0$ monotonically decreases on the positive half-axis, which follows directly from the definition of $K_0$. Figure \ref{fig:ModBessel} illustrates these facts.

We next turn to part (c). Part (a) implies
\begin{align*}
\|f_a\sigma\ast f_a\sigma\ast f_a\sigma\|_{L^2(\R^2)}^2
=\int_{\mathcal{P}_{1,3}} e^{-2a\tau}(\sigma^{(\ast 3)}(\xi,\tau))^2\,\d\xi\,\d\tau,
\end{align*}
where the support region $\mathcal{P}_{1,3}$ was defined in \eqref{SuppConvolution}.
We perform the change of variables $\phi(\xi,\tau)=(\xi,\sqrt{\tau^2+\xi^2})$, which has Jacobian determinant
\begin{displaymath}
J(\phi)(\xi,\tau)=\det\left(
\begin{array}{cc}
1 & \frac{\xi}{\sqrt{\tau^2+\xi^2}}\\
0 & \frac{\tau}{\sqrt{\tau^2+\xi^2}}
\end{array}\right)
=\frac{\tau}{\sqrt{\tau^2+\xi^2}}.
\end{displaymath}
As a consequence,
\begin{align}
\|f_a\sigma\ast f_a\sigma\ast f_a\sigma\|_{L^2(\R^2)}^2
&=\int_{\phi^{-1}(\mathcal{P}_{1,3})} e^{-2a\sqrt{\tau^2+\xi^2}}\Big(\sigma^{(\ast 3)}\big(\xi,\sqrt{\tau^2+\xi^2}\big)\Big)^2 \frac{\tau}{\sqrt{\tau^2+\xi^2}} \,\d\xi\,\d\tau\notag\\
&=\int_3^{\infty} \tau\big(\sigma^{(\ast 3)}(0,\tau)\big)^2 \Big(\int_\R\frac{ e^{-2a\sqrt{\tau^2+\xi^2}}}{\sqrt{\tau^2+\xi^2}} \d\xi\Big)\d\tau,\label{PreH}
\end{align}
where in the last identity we used the Lorentz invariance \eqref{LorentzInv} of the convolution $\sigma^{(\ast 3)}$, together with the fact that $\phi^{-1}(\mathcal{P}_{1,3})=\R\times (3,\infty).$ 
Define $H(a):=\|f_a\|_{L^2(\H^1)}^2$.
Recognizing the inner integral in \eqref{PreH} as the quantity $H(a\tau)$, we have that
\begin{equation}\label{July13_11:00am}
\frac{\|f_a\sigma\ast f_a\sigma\ast f_a\sigma\|_{L^2(\R^2)}^2}{\|f_a\|_{L^2(\H^1)}^6}=\int_3^{\infty} \tau(\sigma^{(\ast 3)}(0,\tau))^2 \frac{H(a\tau)}{H(a)^3}\,\d\tau.
\end{equation}
We will be interested in the regime where $a\to\infty$, for which the approximation 
\begin{equation}\label{July13_11:01am}
\lim_{a \to \infty} \bigg(\frac{H(a\tau)}{H(a)^3} \bigg) \,\bigg/\, \bigg(\frac a{\pi}\frac{e^{-2a(\tau-3)}}{\sqrt{\tau}}\bigg) = 1
\end{equation}
follows from \eqref{HrelK} and \eqref{KBesselAtInfty}.
On the other hand, we have noted in the course of the proof of Lemma \ref{Sigma3Properties} that
$$\sigma^{(\ast 3)}(0,\tau)\lesssim\frac1{\tau},\text{ for }\tau >3.$$
From this and support considerations, it follows that the function $\tau\mapsto\sqrt{\tau} (\sigma^{(\ast 3)}(0,\tau))^2$ is bounded on the positive half-axis. It is also continuous there, except for a jump discontinuity at $\tau=3$.
Given an even function $\varphi:\R\to[0,\infty)$ satisfying $\int_\R \varphi=1$, we have
\begin{align*}
\lim_{a\to\infty}\int_\R \sqrt{|\tau|} (\sigma^{(\ast 3)}(0,\tau))^2 \,a\,\varphi(a(\tau-3))\,\d \tau
&=\sqrt{3}\,\frac{(\sigma^{(\ast 3)}(0,3^-))^2+(\sigma^{(\ast 3)}(0,3^+))^2}{2}\\
&=\sqrt{3}\,\frac{0^2+(\frac{2\pi}{\sqrt{3}})^2}{2}
=\frac{2\pi^2}{\sqrt{3}}.
\end{align*}
This follows from the fact that $\{a\varphi(a\;\cdot)\}_{a\in\N}$ constitutes an approximate identity sequence, as $a\to\infty.$
Specializing to $\varphi(t)=e^{-2|t|}$, and using \eqref{July13_11:00am} and \eqref{July13_11:01am}, we check that \eqref{PreBestConstant} holds. From \eqref{BestConstant1D} and \eqref{July13_12:58pm} it follows that the sequence $\{f_a\}_{a\in\N}$ is extremizing for inequality \eqref{ExtensionInequality}, and
\begin{align*}
{\bf H}_{1,6} = 3^{-\frac1{12}}(2\pi)^{\frac12}.
\end{align*}
This completes the proof of the proposition (and of part of Theorem \ref{Thm1}).
\end{proof}
\begin{figure}
  \centering
  \includegraphics[height=4cm]{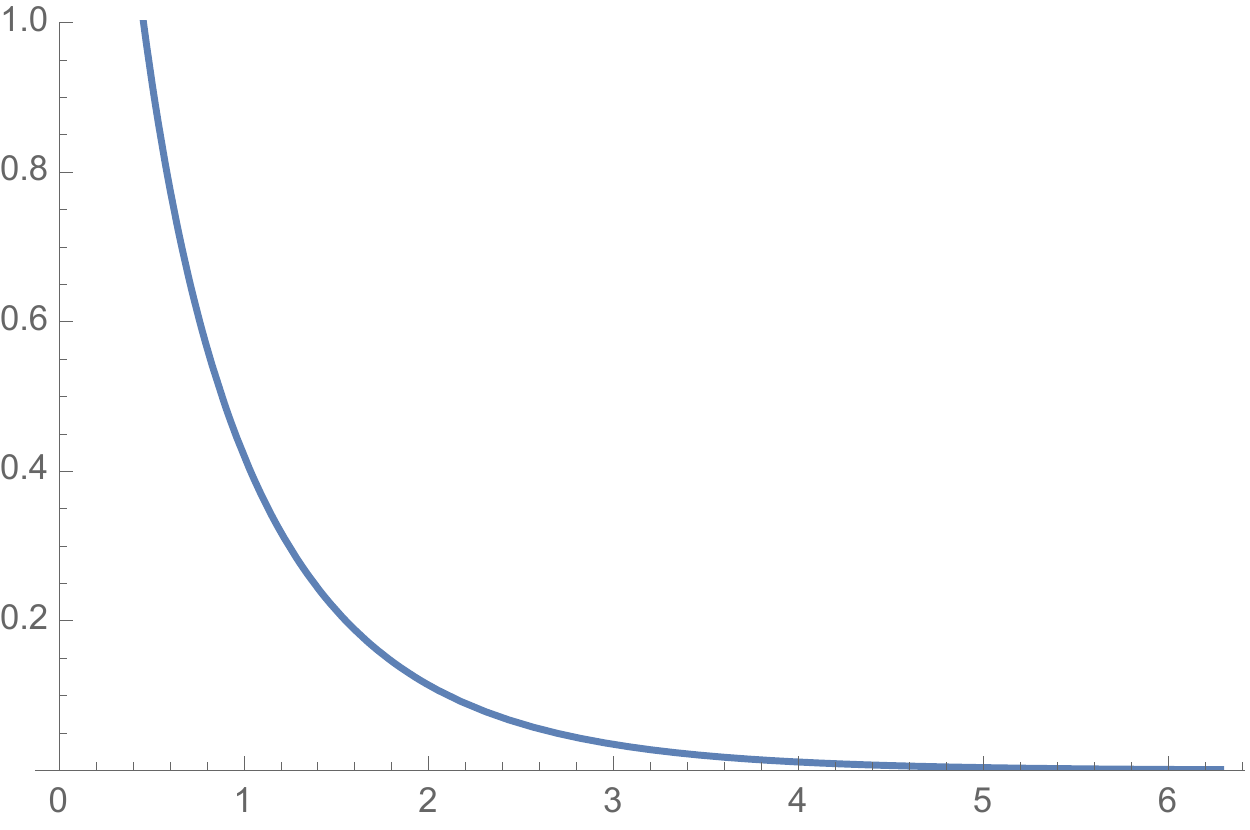} \ \ \ \ 
    \includegraphics[height=4cm]{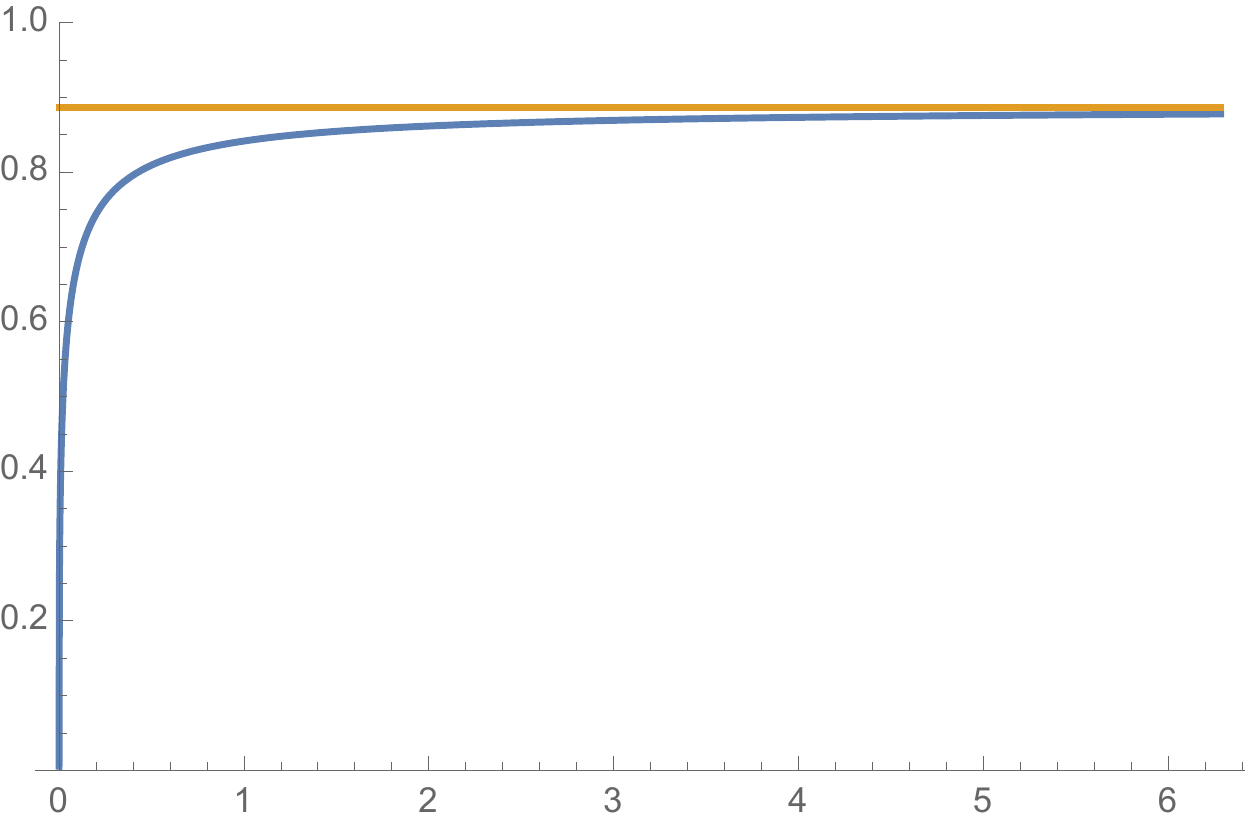} 
    \caption{The function $x\mapsto K_0(x)$, and the function $x\mapsto \sqrt{x}e^{2x}K_0(2x)$ together with its horizontal asymptote $y=\frac{\sqrt{\pi}}2$.}
\label{fig:ModBessel}
\end{figure}
To prove that extremizers do not exist, we invoke the useful observation from \cite[Corollary 4.3]{Qu15}, which we record here.
\begin{lemma}[cf. \cite{Qu15}]\label{QuilodranLemmaBestConstant}
Let $(d,p)$ satisfy \eqref{AdmissibleRange}, and suppose that $p=2n$ is an even integer. 
Suppose that 
$${\bf H}_{d,p}=(2\pi)^{\frac{d+1}p}\|\sigma^{(\ast n)}\|_{L^\infty(\R^{d+1})}^{\frac 1p},$$ 
and that $\sigma^{(\ast n)}(\xi,\tau)<\|\sigma^{(\ast n)}\|_{L^\infty(\R^{d+1})}$ for a.e. $(\xi,\tau)$ in the support of $\sigma^{(\ast n)}$. 
Then extremizers for inequality \eqref{ExtensionInequality} do not exist for the pair $(d,p)$.
\end{lemma}
Armed with Lemma \ref{Sigma3Properties}, Proposition \ref{ExtremizingSequence} ( \!\!c) and Lemma \ref{QuilodranLemmaBestConstant}, it is now an easy matter to finish the proof of Theorem \ref{Thm1}.

We end this section with the following remark. The extremizing sequence $\{f_a\}_{a\in\N}$ defined in the statement of Proposition \ref{ExtremizingSequence} concentrates at the vertex of the hyperbola. It is sensible to ask about the behaviour of general extremizing sequences. Following the arguments from \cite[\S 6]{Qu15} or \cite[Theorem 1.5]{OSQ16}, one can show that {\it every} extremizing sequence for inequality \eqref{ExtensionInequality} when $(d,p)=(1,6)$ concentrates at the vertex of the hyperbola, possibly after applying the symmetries of the problem and after extracting a subsequence. We omit the details.

\section{Special cap}\label{sec:SpecialCap}
In this section, we seek to locate a distinguished cap which carries a non-trivial amount of $L^2$ mass.
This is essential to start gaining some control on compactness properties of extremizing sequences. 
In the one-dimensional situation, we establish a refinement of the Fourier extension inequality. 
In the two-dimensional setting, we reduce matters to the study of bilinear interactions in the lower endpoint case.
\subsection{One-dimensional setting}
This subsection in partially inspired by \cite[\S 3]{Qu13} (see also \cite[\S 4]{KSV12}). To study the interaction between the distinct caps from the family $\{\C_k\}_{k\in\Z}$, defined in \eqref{July21_7:34pm}, we make use of the following standard result on fractional integration.
\begin{lemma}[Hardy--Littlewood--Sobolev]\label{HLS}
Given $r,s\in(1,\infty)$ with $\frac1r+\frac1s>1$, let $\alpha=2-\frac1r-\frac1s$. For any $f\in L^r(\R)$ and $g\in L^s(\R)$,
$$\int_{-\infty}^\infty\int_{-\infty}^\infty |f(x)| |x-y|^{-\alpha} |g(y)|\, \d x\,\d y \lesssim_{r,s} \|f\|_{L^r(\R)}\|g\|_{L^s(\R)}.$$
\end{lemma}
The following result shows that distant caps interact weakly.
\begin{lemma}\label{WeakInteractionLemma}
Let $2<q<\infty$. For any $k,\ell\in\Z$, if ${\rm supp} (f)\subset \C_k$ and ${\rm supp} (g)\subset \C_\ell$, then
\begin{equation*}
\|Tf\cdot Tg\|_{L^q(\R^2)} \lesssim_{q}\, e^{-\frac{|k-\ell|}{2q}} \|f\|_{L^{\frac{2q}{2q-3}}(\H^1)}\|g\|_{L^{\frac{2q}{2q-3}}(\H^1)}.
\end{equation*}
\end{lemma}
\begin{proof}
Define the auxiliary function
$$h(\xi_1,\xi_2):=(f(\xi_1)g(\xi_2)+f(\xi_2)g(\xi_1))\one_{\{\xi_1\geq\xi_2\}}(\xi_1,\xi_2),$$
for which
$$(Tf\cdot Tg)(x,t)
=\int_{\R^2} e^{ix(\xi_1+\xi_2)}e^{it(\langle\xi_1\rangle+\langle\xi_2\rangle)}h(\xi_1,\xi_2)\frac{\d\xi_1}{\langle\xi_1\rangle}\frac{\d\xi_2}{\langle\xi_2\rangle}.$$
Following an argument that goes back to early work of Carleson--Sj\"olin \cite{CS72}, we change variables
$$(\xi_1,\xi_2)\mapsto(u,v)=(\xi_1+\xi_2,\langle\xi_1\rangle+\langle\xi_2\rangle)$$
in the region of integration $\{(\xi_1,\xi_2)\in\R^2: \xi_1\geq \xi_2\}$. Note that this is a bijective map onto the region $\{(u,v)\in\R \times \R^{+}: v^2 - u^2 \geq 4\}$. It follows that
$$(Tf\cdot Tg)(x,t)
=\int_{\R^2} e^{i(x,t)\cdot(u,v)}\frac{h(\xi_1,\xi_2)}{J(\xi_1,\xi_2)}\frac{\d u}{\langle\xi_1\rangle}\frac{\d v}{\langle\xi_2\rangle},$$
where $J$ denotes the Jacobian of this transformation, given by
$$J(\xi_1,\xi_2)=\left|\frac{\partial(u,v)}{\partial(\xi_1,\xi_2)}\right|
=\left|\frac{\xi_1}{\langle\xi_1\rangle}-\frac{\xi_2}{\langle\xi_2\rangle}\right|.$$
The Hausdorff--Young inequality on $\R^2$ implies that, for every $q\geq 2$,
$$\|Tf\cdot Tg\|_{L^q(\R^2)}\leq
\left(\int_{\R^2} \left|\frac{h(\xi_1,\xi_2)}{J(\xi_1,\xi_2)}\frac{1}{\langle\xi_1\rangle\langle\xi_2\rangle}\right|^{\frac q{q-1}}\d u\,\d v\right)^{\frac{q-1}q}.$$
Changing back to the original variables $(\xi_1,\xi_2)$, we obtain
$$\|Tf\cdot Tg\|_{L^q(\R^2)}\leq
\left(\int_{\R^2}  \frac{|h(\xi_1,\xi_2)|^{\frac q{q-1}}}{(J(\xi_1,\xi_2)\langle\xi_1\rangle\langle\xi_2\rangle)^{\frac 1{q-1}}}\frac{\d \xi_1}{\langle\xi_1\rangle}\frac{\d \xi_2}{\langle\xi_2\rangle}\right)^{\frac{q-1}q}.$$

In order to invoke Lemma \ref{HLS}, it is convenient to perform another change of variables $(\xi_1,\xi_2)=(\sinh(\theta_1),\sinh(\theta_2))$. 
Noting that
$$J(\xi_1,\xi_2)\langle\xi_1\rangle\langle\xi_2\rangle=|\sinh(\theta_1-\theta_2)|,$$
 Minkowski's inequality yields
 \begin{align}\label{July14_17:46pm}
 \begin{split}
 \|Tf\cdot Tg\|_{L^q(\R^2)}
 &\leq
 \left(\int_{\R^2}  \frac{|h(\sinh(\theta_1),\sinh(\theta_2))|^{\frac q{q-1}}}{|\sinh(\theta_1-\theta_2)|^{\frac 1{q-1}}}\,\d\theta_1\,\d\theta_2\right)^{\frac{q-1}q}\\
 &\lesssim
  \left(\int_{\R^2}  \frac{|f(\sinh(\theta_1))g(\sinh(\theta_2))|^{\frac q{q-1}}}{|\sinh(\theta_1-\theta_2)|^{\frac 1{q-1}}}\,\d\theta_1\,\d\theta_2\right)^{\frac{q-1}q}.
  \end{split}
 \end{align}
 
We now use the lower bound
 $$|\sinh(\theta)|\geq \frac{|\theta|}2 \exp\left(\frac{|\theta|}2\right),$$ 
which is valid for any $\theta\in\R$, together with Lemma \ref{HLS} with the choices $\alpha=\frac{1}{q-1}$ and $r=s=\frac{2q-2}{2q-3}$ (which are admissible since $2 < q < \infty$). From \eqref{July14_17:46pm} we get
 \begin{align*}
 \|Tf\cdot Tg\|_{L^q(\R^2)}
 &\lesssim
  \left(\int_{\R^2}  \frac{|f(\sinh(\theta_1))g(\sinh(\theta_2))|^{\frac q{q-1}}}{|\theta_1-\theta_2|^{\frac 1{q-1}}\,e^{\frac{|\theta_1-\theta_2|}{2(q-1)}}}\,\d\theta_1\,\d\theta_2\right)^{\frac{q-1}q}\\
 &\lesssim_q
e^{-\frac{|k-\ell|}{2q}}  \|f\|_{L^{\frac{2q}{2q-3}}(\H^1)}
\|g\|_{L^{\frac{2q}{2q-3}}(\H^1)},
\end{align*} 
where we have used that $|(\theta_1-\theta_2)-(k-\ell)|\leq 1$, since $|\theta_1-k|\leq \frac12$ and $|\theta_2-\ell|\leq \frac12$.
This completes the proof of the lemma.
 \end{proof}

Lemma \ref{WeakInteractionLemma} is the key ingredient in order to establish the following refinement of the inequality \eqref{ExtensionInequality} in the case $d=1$. 
In what follows, given $f\in L^2(\H^1)$, we shall decompose $f=\sum_{k\in\Z} f_k$ with $f_k:=f\one_{\C_k}$.

\begin{proposition}\label{BesovProp}
Let $3 \leq q < \infty$. For any $f\in L^2(\H^1)$ we have
\begin{equation}\label{PreCapBound}
\|T f\|_{L^{2q}(\R^2)} \lesssim_q \left(\sum_{\,k\in\Z} \|f_k\|_{L^{\frac{2q}{2q-3}}(\H^1)}^3\right)^{\frac13}.
\end{equation}
\end{proposition}
When $3 \leq q \leq \infty$, note that $1 \leq \frac{2q}{2q-3}\leq2$. 
In this case,  the estimates 
$$\|Tf\|_{L^6(\R^2)}\lesssim \|f\|_{L^2(\H^1)},\text{ and } \|Tf\|_{L^\infty(\R^2)}\leq \|f\|_{L^1(\H^1)}$$
can be interpolated to yield
\begin{equation}\label{InterpolatedExtension}
\|Tf\|_{L^{2q}(\R^2)}\lesssim_q \|f\|_{L^{\frac{2q}{2q-3}}(\H^1)}.
\end{equation}
\begin{proof}[Proof of Proposition \ref{BesovProp}]
Writing
$$(Tf)^3=\sum_{k,\ell,m\in\Z} Tf_{k} \cdot Tf_{\ell}\cdot Tf_{m},$$
Minkowski's triangle inequality plainly implies that
\begin{equation}\label{FirstBound}
\|Tf\|_{L^{2q}(\R^2)}^3=\big\|(Tf)^3\big\|_{L^{\frac{2q}3}(\R^2)}
\leq \sum_{k,\ell,m\in\Z} \| Tf_{k} \cdot Tf_{\ell}\cdot Tf_{m}\|_{L^{\frac{2q}3}(\R^2)}.
\end{equation}
Given a  triplet $(k,\ell,m)\in\Z^3$, we lose no generality in assuming that
$$|k-\ell|=\max\{|k-\ell|,|\ell-m|,|m-k|\}.$$
 H\"older's inequality, Lemma \ref{WeakInteractionLemma} (recall that $q < \infty$) and estimate \eqref{InterpolatedExtension}, together with the maximality of $|k-\ell|$, imply
\begin{align}\label{SecondBound}
\begin{split}
\| Tf_{k} \cdot Tf_{\ell}& \cdot Tf_{m}\|_{L^{\frac{2q}3}(\R^2)}
\leq \|Tf_k\cdot Tf_\ell\|_{L^{q}(\R^2)}\|Tf_m\|_{L^{2q}(\R^2)}\\
&\lesssim_q e^{-\frac{|k-\ell|}{2 q}} \|f_k\|_{L^{\frac{2q}{2q-3}}(\H^1)}\|f_\ell\|_{L^{\frac{2q}{2q-3}}(\H^1)}\|f_m\|_{L^{\frac{2q}{2q-3}}(\H^1)}\\
&\leq e^{-\frac{|k-\ell|}{6 q}}e^{-\frac{|\ell-m|}{6q}}e^{-\frac{|m-k|}{6 q}} \|f_k\|_{L^{\frac{2q}{2q-3}}(\H^1)}\|f_\ell\|_{L^{\frac{2q}{2q-3}}(\H^1)}\|f_m\|_{L^{\frac{2q}{2q-3}}(\H^1)}.
\end{split}
\end{align}
Putting together \eqref{FirstBound} and \eqref{SecondBound}, we conclude that 
\begin{align*}
\|Tf\|_{L^{2q}(\R^2)}^3
&\lesssim_q
 \sum_{k,\ell,m\in\Z} e^{-\frac{|k-\ell|}{6 q}}e^{-\frac{|\ell-m|}{6q}}e^{-\frac{|m-k|}{6 q}} \|f_k\|_{L^{\frac{2q}{2q-3}}(\H^1)}\|f_\ell\|_{L^{\frac{2q}{2q-3}}(\H^1)}\|f_m\|_{L^{\frac{2q}{2q-3}}(\H^1)}.\\
 &\leq \sum_{k,\ell,m\in\Z} e^{-\frac{|k-\ell|}{6 q}}e^{-\frac{|\ell-m|}{6q}}e^{-\frac{|m-k|}{6 q}} \|f_k\|^3_{L^{\frac{2q}{2q-3}}(\H^1)},
 \end{align*}
where the last line follows from the inequality between the arithmetic and the geometric means. Summing two geometric series, we finally have that
 $$\|Tf\|_{L^{2q}(\R^2)}^3\lesssim_q\sum_{k\in\Z}\|f_k\|^3_{L^{\frac{2q}{2q-3}}(\H^1)},$$
 as desired. This completes the proof of the proposition.
\end{proof}

We have the following immediate but useful consequence.
\begin{corollary}
Let $6\leq p < \infty$. Then there exists $C_p<\infty$ such that, for any $f\in L^2(\H^1),$
\begin{equation}\label{CapBound}
\|Tf\|_{L^p(\R^2)}
\leq C_p \, \sup_{k\in\Z} \|f_k\|^{\frac13}_{L^2(\H^1)}\,
 \|f\|_{L^2(\H^1)}^{\frac23}.
 \end{equation}
\end{corollary}
\begin{proof}
Using estimate \eqref{PreCapBound} with $p=2q$ we get
\begin{equation}\label{Pre2CapBound}
\|Tf\|_{L^p(\R^2)} \leq C_p \, \sup_{k\in\Z} \|f_k\|^{\frac13}_{L^{\frac{p}{p-3}}(\H^1)}
\left(\sum_{k\in\Z} \|f_k\|_{L^{\frac p{p-3}}(\H^1)}^2\right)^{\frac 13},
\end{equation}
for some constant $C_p < \infty$. Applying H\"older's inequality, and recalling that $\sigma(\C_k)=1$, 
$$\|f_k\|_{L^{\frac{p}{p-3}}(\H^1)}
\leq \|\one_{\C_k}\|_{L^{\frac{2p}{p-6}}(\H^1)}\|f_k\|_{L^{2}(\H^1)}
=\|f_k\|_{L^{2}(\H^1)}.$$
Plugging this into the right-hand side of \eqref{Pre2CapBound}, and appealing to the disjointness of the supports of the $\{f_k\}$, yields the desired conclusion.
\end{proof}

\begin{proposition}\label{SpecialCap1D}
Let $d=1$ and $6\leq p < \infty$. Let $\{f_n\}_{n\in\N}\subset L^2(\H^1)$ be an extremizing sequence for inequality \eqref{ExtensionInequality}, normalized so that $\|f_n\|_{L^2(\H^1)}=1$ for each $n \in \N$. There exists a universal constant $\eta_{1,p}>0$ and $n_0 \in \N$, such that for any $n \geq n_0$ there exists $s_n\in (-1,1)$ verifying
$$\|(L^{s_n})^*f_n\|_{L^2(\C_0)}
\geq\frac12\,\sup_{k\in\Z} \|(L^{s_n})^* f_n\|_{L^2(\C_k)}
\geq \eta_{1,p}.$$
\end{proposition}

\begin{proof} Let $n_0 \in \N$ be such that, for $n \geq n_0$, we have 
\begin{equation}\label{July15_1:00pm}
\|T f_n\|_{L^p(\R^2)}\geq\frac{{\bf H}_{1,p}}2.
\end{equation}
Fix $n\geq n_0$. Using \eqref{July15_1:00pm} and the cap bound \eqref{CapBound}, we have that 
$$\sup_{k\in\Z} \|f_{n,k}\|_{L^2(\H^1)}\geq\left(\frac{{\bf H}_{1,p}}{2 C_p}\right)^3,$$
where $f_{n,k}:=f_n\one_{\C_k}$, and  $C_p$ is the constant from inequality \eqref{CapBound}. 
Let $k(n)\in\Z$ be such that
$$\|f_{n,k(n)}\|_{L^2(\H^1)}\geq \eta_{1,p}:=\frac12\left(\frac{{\bf H}_{1,p}}{2 C_p}\right)^3.$$
Choose $s_n:=\tanh(k(n))$. Since $L^{s_n}(\C_{0})=\C_{k(n)}$ (recall Lemma \ref{1Dtessellation} (b)), we then have by \eqref{NormInvariance} that
$$\text{supp}\big((L^{s_n})^*f_{n,k(n)}\big)\subseteq \C_0,\text{ and }\|(L^{s_n})^*f_{n,k(n)}\|_{L^2(\H^1)}\geq \eta_{1,p}.$$
This concludes the proof of the proposition.
\end{proof}

\subsection{Two-dimensional setting}
In order to study the interaction between the distinct caps from the family $\{\C_{n,j}\}$ defined in \eqref{2DCap} and \eqref{2DCap0}, we try to relate the non-endpoint problem to the lower endpoint problem. Log-convexity of Lebesgue norms readily implies the following: given $p\in (4,6)$, there exists $\theta\in(0,1)$ such that
\begin{equation}\label{LogConvexity}
\|Tf\|_{L^p(\R^3)}\leq \|Tf\|_{L^4(\R^3)}^\theta \|Tf\|_{L^6(\R^3)}^{1-\theta}.
\end{equation}
In particular, if $\{f_n\}_{n\in\N}$ is an extremizing sequence for inequality \eqref{ExtensionInequality} when $d=2$ and $p\in (4,6)$, normalized so that $\|f_n\|_{L^2(\H^2)}=1$ for each $n \in \N$, then both quantities on the right-hand side of inequality \eqref{LogConvexity} cannot be too small, in the sense that there exists a universal constant $\gamma_p>0$, depending on $p$ but {\it not} on $n$, and $n_0 \in \N$ such that
$$\min\{\|Tf_n\|_{L^4(\R^3)}, \|Tf_n\|_{L^6(\R^3)}\}\geq \gamma_p, \text{ for any } n \geq n_0.$$
The idea will be to exploit the convolution structure of the lower endpoint problem $(d,p)=(2,4)$ to derive some nontrivial information about the non-endpoint case. 
The crux of the matter lies in the following result.

\begin{proposition}\label{SpecialCap2DMainProp}
Let  $f\in L^2(\H^2)$, normalized so that $\|f\|_{L^2(\H^2)}=1$. Let $\eps>0$ and assume that 
$$\sup_{n,j} \|f\|_{L^2(\C_{n,j})}\leq \eps,$$
where the supremum is taken over all $n\in\N \cup \{0\}$ and $0\leq j< 2^n$. Then
\begin{equation}\label{July15_8:07pm}
\|Tf\|^4_{L^4(\R^3)}\lesssim \eps\log_2(\eps^{-1}).
\end{equation}
\end{proposition}
\noindent {\it Remark:} The relevant feature of the function $\Phi(\eps)=\eps\log_2(\eps^{-1})$ is that $\Phi(\eps)\to 0$, as $\eps\to 0^+$. Any other $\Phi$ with the same property would serve our  purpose equally well.
\begin{proof}
Recalling \eqref{TintermsofHat}, the usual application of Plancherel's Theorem and the Cauchy--Schwarz inequality (the latter as in \eqref{July13_13:01pm}) yields
$$\|Tf\|_{L^4(\R^3)}^4
\simeq \|f\sigma\ast f\sigma\|_{L^2(\R^3)}^2
\leq \int_{\R^2\times\R^2} 
|f(\xi)|^2|f(\eta)|^2 (\sigma\ast\sigma)(\xi+\eta,\langle\xi\rangle+\langle\eta\rangle)\frac{\d \xi}{\langle\xi\rangle}\frac{\d \eta}{\langle\eta\rangle}.$$
Abusing notation slightly, and still denoting by $\{\C_{n,j}\}$ the projection of the caps defined in \eqref{2DCap} and \eqref{2DCap0} onto the $\xi$-plane, we have that
\begin{equation}\label{SumOverCaps}
\|Tf\|_{L^4(\R^3)}^4
\lesssim
\sum_{(m,\ell),(n,j)} \int_{\C_{m,\ell}} \int_{\C_{n,j}}
|f(\xi)|^2|f(\eta)|^2 (\sigma\ast\sigma)(\xi+\eta,\langle\xi\rangle+\langle\eta\rangle)\frac{\d \xi}{\langle\xi\rangle}\frac{\d \eta}{\langle\eta\rangle}.
\end{equation}
Here, the sum is taken over all pairs $(m,\ell),(n,j)$ with $0\leq n\leq m$, and $0\leq \ell<2^m$, $0\leq j<2^n$.
We seek to obtain some control over the height $s$, defined via the equation
\begin{equation}\label{heights}
s^2:=(\langle\xi\rangle+\langle\eta\rangle)^2-|\xi+\eta|^2=2(1+\langle\xi\rangle\langle\eta\rangle-\xi\cdot\eta).
\end{equation}
With this purpose in mind, we split the sum in \eqref{SumOverCaps} into two pieces, depending on whether or not the direction of the caps is approximately the same. In the former case, the bound will be in terms of the distance between the centers of the caps, whereas in the latter case one obtains an improved bound in terms of the angular separation between the caps. See Figure \ref{fig:Cap_Interaction_Slice} for an illustration of two extreme cases of this separation.

Let $S \subset \R$. In what follows, we say that $x \in S \, {\rm mod} \, 1$ if $x + k \in S$ for some $k \in \Z$. Analogously, for $m \in \N \cup \{0\}$, we say that $x \in S \, {\rm mod} \, 2^m$ if $x + k2^m \in S$ for some $k \in \Z$. We also define
\begin{equation}\label{July23_8:33am}
\|x\|:= \min\{|x-k| : k \in \Z\}
\end{equation}
for the distance of $x$ to the nearest integer.
\smallskip

\noindent {\it Case 1. $\frac{\ell}{2^m} - \frac{j}{2^n} \in \left[ -\frac2{2^n}, \frac2{2^n}\right) \, {\rm mod} \, 1$.}
In this case, we are considering indices $\ell$ belonging to 
$$A_{m,n}^{(j)}:=\left\{0 \leq \ell < 2^m : \ell \in \left[2^{m-n}(j-2) , 2^{m-n}(j+2)\right) \, {\rm mod} \,2^m\right\},$$
a set of cardinality $\big|A_{m,n}^{(j)}\big|=2^{m-n+2}$. We seek to estimate the sum
$$S_1:=\sum_{n\geq 0}\sum_{m\geq n}\sum_{0\leq j<2^n}\sum_{\ell\in A_{m,n}^{(j)}} \int_{\C_{m,\ell}} \int_{\C_{n,j}}
|f(\xi)|^2|f(\eta)|^2 (\sigma\ast\sigma)(\xi+\eta,\langle\xi\rangle+\langle\eta\rangle)\frac{\d \xi}{\langle\xi\rangle}\frac{\d \eta}{\langle\eta\rangle}.$$
Note that $x\mapsto{\langle x\rangle}{|x|^{-1}}$ is a decreasing function of $|x|$. 
For $\xi\in\C_{n,j}$ and $\eta\in\C_{m,\ell}$, we can estimate the height $s$ defined in \eqref{heights} from below, as follows:
$$s^2\gtrsim \langle\xi\rangle\langle\eta\rangle-\xi\cdot\eta
\geq |\xi||\eta|\left(\frac{\langle\xi\rangle\langle\eta\rangle}{|\xi||\eta|}-1\right)
\geq \frac14\, 2^{n+1} \,2^{m+1}\left(\frac{\langle2^{n+1}\rangle\langle2^{m+1}\rangle}{2^{n+1}2^{m+1}}-1\right).$$
Writing $2^{n+1}=\sinh(\sinh^{-1}(2^{n+1}))$ and $\langle2^{n+1}\rangle=\cosh(\sinh^{-1}(2^{n+1}))$, and similarly for $m$, we have that
$$s^2\gtrsim\cosh\big(\sinh^{-1}(2^{m+1})-\sinh^{-1}(2^{n+1})\big).$$
Since $\sinh^{-1}(x)=\log(x+\sqrt{x^2+1})$ and $\cosh(x)\gtrsim\exp(x)$, we can further estimate
\begin{align*}
s^2
&\gtrsim\cosh(\log(2^{m+1}+\sqrt{2^{2m+2}+1})-\log(2^{n+1}+\sqrt{2^{2n+2}+1}))\\
&\gtrsim\exp(\log(2^{m+1}+\sqrt{2^{2m+2}+1})-\log(2^{n+1}+\sqrt{2^{2n+2}+1}))\\
&=\frac{2^{m+1}+\sqrt{2^{2m+2}+1}}{2^{n+1}+\sqrt{2^{2n+2}+1}}\gtrsim 2^{m-n}.
\end{align*}

Under the same assumptions on $\xi,\eta$, it follows from  the Lorentz invariance \eqref{LorentzInv} and Lemma \ref{2DConvolutions} (a) that 
$$(\sigma\ast\sigma)(\xi+\eta,\langle\xi\rangle+\langle\eta\rangle)
=(\sigma\ast\sigma)(0,s)\lesssim s ^{-1} \lesssim2^{-\frac{m-n}2}.$$
The sum $S_1$ can then be estimated  by
$$S_1\lesssim\sum_{n\geq 0}\sum_{m\geq n}\sum_{0\leq j<2^n}  \frac{\|f\|_{L^2(\C_{n,j})}^2}{2^{\frac{m-n}2}}\Bigg(\sum_{\ell\in A_{m,n}^{(j)}}\|f\|_{L^2(\C_{m,\ell})}^2\Bigg),$$
where the inner sum is trivially bounded by
$$\sum_{\ell\in A_{m,n}^{(j)}}\|f\|_{L^2(\C_{m,\ell})}^2\leq\min\big\{1,\eps^2 \,2^{m-n+2}\big\}.$$
It follows that
\begin{equation}\label{EstimatingS1}
S_1\lesssim \sum_{n\geq 0}\|f\|_{L^2(\mathcal{A}_{n})}^2
\left(\sum_{m\geq n}  \frac{\min\{1,\eps^2\, 2^{m-n+2}\}}{2^{\frac{m-n}2}}\right),
\end{equation}
where ${\mathcal{A}_n}$ denotes the annulus defined in \eqref{DefAnnulus}.
We  estimate the inner sum on the right-hand side of \eqref{EstimatingS1} by breaking it up in two pieces, according to whether or not the integer $\kappa:=m-n+2$ satisfies $\eps^2 2^{\kappa}< 1$, or equivalently $\kappa<2\log_2(\eps^{-1})$.
We obtain
\begin{align*}
\sum_{m\geq n}  \frac{\min\{1,\eps^2 \,2^{m-n+2}\}}{2^{\frac{m-n}2}}
=\sum_{\kappa<2\log_2(\eps^{-1})}\eps^2 \,2^{\frac {\kappa}2+1}+\sum_{\kappa\geq2\log_2(\eps^{-1})} 2^{-\frac {\kappa}2+1}
\lesssim \eps,
\end{align*}
where both geometric sums were estimated by their largest terms. Plugging this back into \eqref{EstimatingS1}, and recalling that $\|f\|_{L^2(\H^2)}=1$ and that the annuli in the family $\{\A_n\}_{n\in\N}$ are disjoint, we finally obtain $S_1\lesssim\eps$.\\

\noindent {\it Case 2. $\frac{\ell}{2^m} - \frac{j}{2^n} \in \left[-\frac12,\frac12\right) \setminus \left[ -\frac2{2^n}, \frac2{2^n}\right) \, {\rm mod}\, 1$.} Note that this case is non-empty only if $n \geq 3$. Let $\xi\in\C_{n,j}$ and $\eta\in\C_{m,\ell}$.
Setting $\theta_{n,j}:=\arg(\xi)$ and $\theta_{m,\ell}:=\arg(\eta)$, we note that, since $n\leq m$,
\begin{equation}\label{alphaminusbeta}
\left|\frac{\theta_{n,j}-\theta_{m,\ell}}{2\pi}-\Big(\frac{j}{2^n}-\frac{\ell}{2^m}\Big)\right| < \frac1{2^{n}}.
\end{equation}

Before we move on, let us make a useful observation. Let 
$$\Gamma_k = \left\{\xi \in \R^2: \frac{k\pi}{2} \leq \arg(\xi) < \frac{(k+1)\pi}{2}\right\}, \  {\rm for} \  0 \leq k\leq 3,$$ 
be the four quadrants of the $\xi$-plane. We may split the function $f$ into four pieces writing $f  = \sum_{k=0}^3 f^{(k)}$, where $f^{(k)} = f\one_{\Gamma_k}$. Since $\|Tf\|_{L^4(\R^3)}^4 \lesssim  \sum_{k=0}^3 \|Tf^{(k)}\|_{L^4(\R^3)}^4$ it suffices to prove \eqref{July15_8:07pm} for each function $f^{(k)}$ separately. {\it In particular, throughout the rest of this proof we may assume that our $f$ is supported in one of the quadrants, say $\Gamma_0$.} Note that this yields $|\theta_{n,j}-\theta_{m,\ell}|\leq\frac{\pi}2$ in the support of $f$ and hence
$$1-\cos(\theta_{n,j}-\theta_{m,\ell})\gtrsim |\theta_{n,j}-\theta_{m,\ell}|^2.$$
As a consequence,
$$s^2\gtrsim \langle\xi\rangle\langle\eta\rangle-\xi\cdot\eta\geq |\xi||\eta|(1-\cos(\theta_{n,j}-\theta_{m,\ell}))
\gtrsim 2^{m+n}|\theta_{n,j}-\theta_{m,\ell}|^2$$
in the support of $f$. Invoking Lemma \ref{2DConvolutions} (a) as before, we have that
\begin{equation}\label{July15_8:48pm}
(\sigma\ast\sigma)(\xi+\eta,\langle\xi\rangle+\langle\eta\rangle)
\lesssim 
2^{-\frac{m+n}2} |\theta_{n,j}-\theta_{m,\ell}|^{-1}.
\end{equation}

We seek to estimate the sum
\begin{equation}\label{July15_8:49pm}
S_2:=\sum_{n\geq 0}\sum_{m\geq n}\sum_{0\leq j<2^n}\sum_{\ell\notin A_{m,n}^{(j)}} \int_{\C_{m,\ell}} \int_{\C_{n,j}}
|f(\xi)|^2|f(\eta)|^2 (\sigma\ast\sigma)(\xi+\eta,\langle\xi\rangle+\langle\eta\rangle)\frac{\d \xi}{\langle\xi\rangle}\frac{\d \eta}{\langle\eta\rangle}.\end{equation}
For fixed indices $0\leq n\leq m$ and $0\leq j<2^n$, we consider the block
\begin{equation}\label{BlocksB}
B_{m,n}^{(j,k)}:=\left\{0 \leq \ell < 2^m : \ell \in \left[2^{m-n}(j+k) , 2^{m-n}(j+k+1)\right) \, {\rm mod} \,2^m\right\},
\end{equation}
for $-2\leq k\leq 2^n-3$. Note that $\big|B_{m,n}^{(j,k)}\big|=2^{m-n}$.
Moreover, since $A_{m,n}^{(j)}$ can be partitioned as a disjoint union,
$$A_{m,n}^{(j)}=B_{m,n}^{(j,-2)}\cup B_{m,n}^{(j,-1)} \cup B_{m,n}^{(j,0)} \cup B_{m,n}^{(j,1)}\,,$$
the fact that $\ell\notin A_{m,n}^{(j)}$ is equivalent to $\ell\in B_{m,n}^{(j,k)}$, for some $2\leq k\leq 2^n-3$. If $\ell\in B_{m,n}^{(j,k)}\setminus A_{m,n}^{(j)}$, then condition \eqref{BlocksB} can be rewritten as
$$\frac{\ell}{2^m}-\frac{j}{2^n} \in \left[\frac{k}{2^n}, \frac{k+1}{2^n}\right)\,{\rm mod}\, 1\,,$$
and it follows from \eqref{alphaminusbeta} that
\begin{equation}\label{July15_8:50pm}
|\theta_{n,j}-\theta_{m,\ell}|\,\gtrsim\left\|\frac{k}{2^n}\right\|,
\end{equation}
where $\|x\|$ was defined in \eqref{July23_8:33am}. Associated to these index blocks, we define the set
$$\mathcal{B}_{m,n}^{(j,k)}:=\bigcup_{\ell\in B_{m,n}^{(j,k)}}\C_{m,\ell}.$$

From \eqref{July15_8:48pm}, \eqref{July15_8:49pm} and \eqref{July15_8:50pm} we get
\begin{align*}
S_2 & \lesssim
\sum_{n\geq 0}\sum_{m\geq n}\sum_{0\leq j<2^n}\sum_{k=2}^{2^n-3}\sum_{\ell\in B_{m,n}^{(j,k)}} \|f\|_{L^2(\C_{m,\ell})}^2\|f\|_{L^2(\C_{n,j})}^2 2^{-\frac{m+n}2} |\theta_{n,j}-\theta_{m,\ell}|^{-1}\\
& \lesssim
\sum_{n\geq 0}\sum_{m\geq n}\sum_{0\leq j<2^n}\sum_{k=2}^{2^n-3} \|f\|_{L^2\big(\mathcal{B}_{m,n}^{(j,k)}\big)}^2\|f\|_{L^2(\C_{n,j})}^2 2^{-\frac{m+n}2}\left\|\frac{k}{2^n}\right\|^{-1}\\
& \leq \sum_{n\geq 0}\sum_{m\geq n}\sum_{0\leq j<2^n}\!\!\!\sum_{k=2}^{2^{n-1}-1} \left( \|f\|_{L^2\big(\mathcal{B}_{m,n}^{(j,k)}\big)}^2 +  \|f\|_{L^2\big(\mathcal{B}_{m,n}^{(j,2^n - k - 1)}\big)}^2\right) \|f\|_{L^2(\C_{n,j})}^2 2^{-\frac{m+n}2} \left(\frac{k}{2^n}\right)^{-1}.
\end{align*}
In order to make use of the trivial bound 
\begin{equation}\label{TrivialBound2}
\|f\|_{L^2\big(\mathcal{B}_{m,n}^{(j,k)}\big)}^2\leq \min\{1,\eps^2 2^{m-n}\},
\end{equation}
we invoke the Cauchy--Schwarz inequality on the innermost sum of
\begin{align*}
S_2&\lesssim
\sum_{n\geq 0}\sum_{m\geq n}2^{-\frac{m-n}2}\sum_{0\leq j<2^n} \|f\|_{L^2(\C_{n,j})}^2
\left(\sum_{k=2}^{2^{n-1}-1} \frac{\|f\|_{L^2\big(\mathcal{B}_{m,n}^{(j,k)}\big)}^2 + \|f\|_{L^2\big(\mathcal{B}_{m,n}^{(j,2^n - k - 1)}\big)}^2}{k} \right)\\
&\lesssim
\sum_{n\geq 0}\sum_{m\geq n}2^{-\frac{m-n}2}\sum_{0\leq j<2^n} \|f\|_{L^2(\C_{n,j})}^2
\Bigg(\sum_{k=2}^{2^n-3} {\|f\|_{L^2\big(\mathcal{B}_{m,n}^{(j,k)}\big)}^4}\Bigg)^{\frac 12}
\Bigg(\sum_{k=2}^{2^{n-1}-1} \frac1{k^2} \Bigg)^{\frac 12}.
\end{align*}
\begin{figure}
  \centering
  \includegraphics[height=6cm]{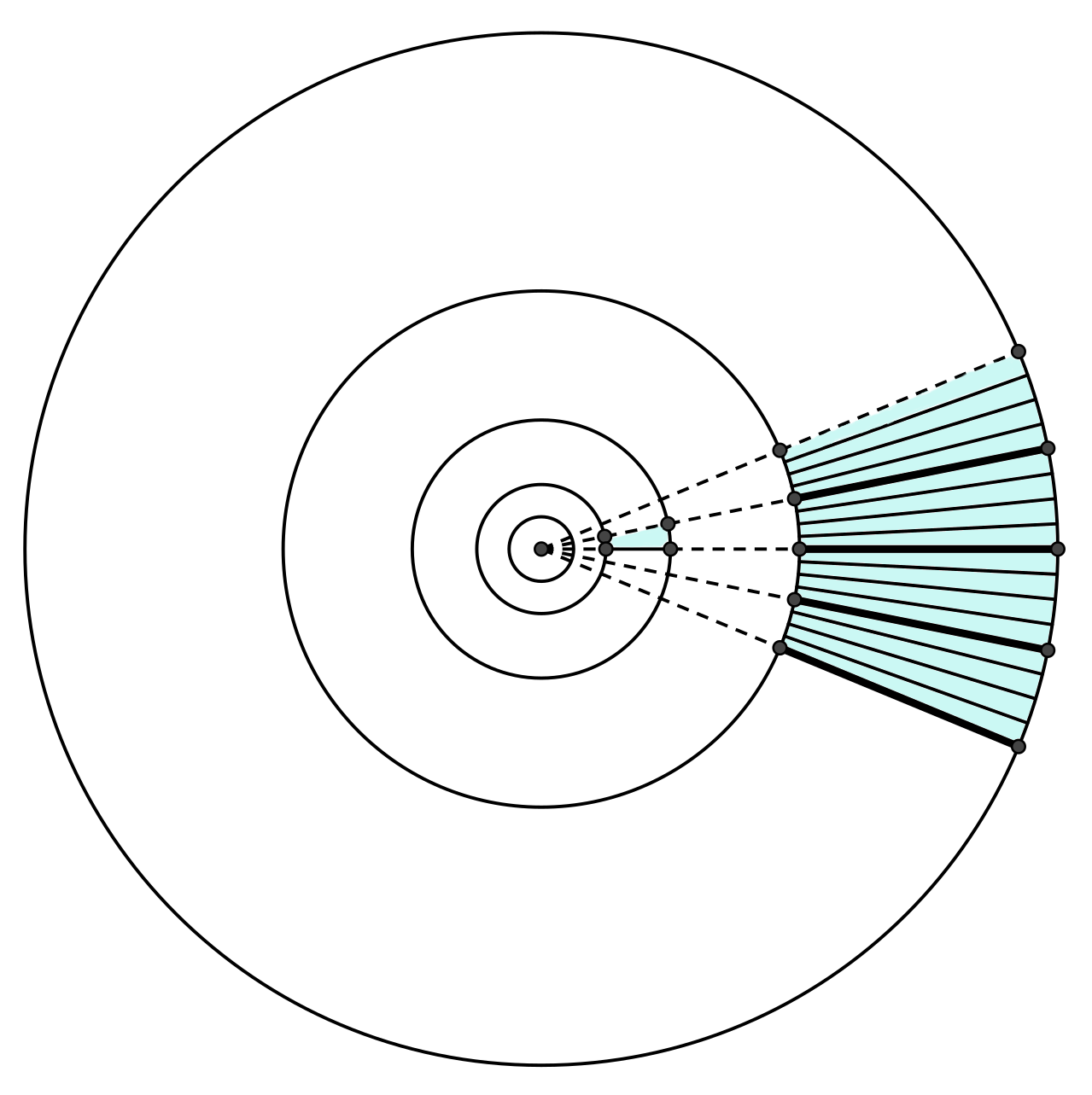} \ \ \ \ \ \ 
    \includegraphics[height=6cm]{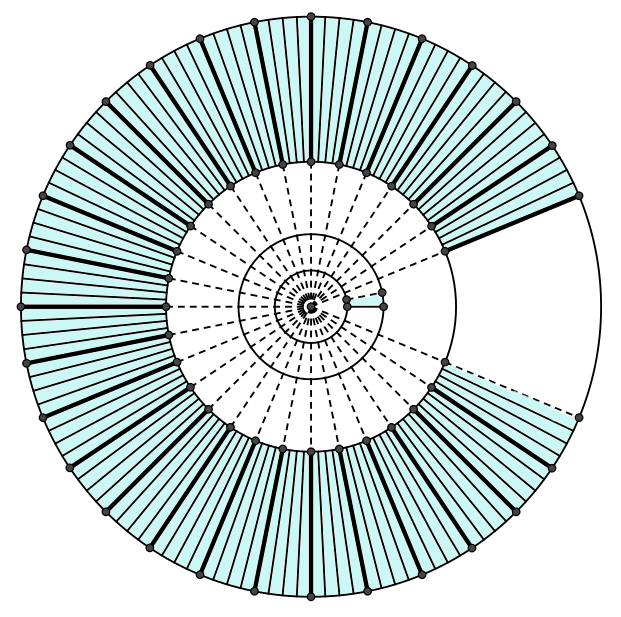} 
    \caption{For a fixed $j$, the regions of Case 1 (on the left) and Case 2 (on the right) of the proof of Proposition \ref{SpecialCap2DMainProp}.}
\label{fig:Cap_Interaction_Slice}
\end{figure}
Recalling \eqref{TrivialBound2}, and noting that the unions
$$\bigcup_{j=0}^{2^n-1} \C_{n,j}=\mathcal{A}_n, \text{ and }\bigcup_{k=2}^{2^n-3}\mathcal{B}_{m,n}^{(j,k)}\subseteq \mathcal{A}_m$$
are disjoint, we have that
$$S_2\lesssim
\sum_{n\geq 0}\sum_{m\geq n}2^{-\frac{m-n}2}
\min\{1,\eps\, 2^{\frac{m-n}2}\}
\|f\|_{L^2(\mathcal{A}_m)}
\|f\|_{L^2(\mathcal{A}_{n})}^2.
$$
We use $\|f\|_{L^2(\mathcal{A}_m)}\leq1$ and estimate the inner sum on the right-hand side of
\begin{equation}\label{EstimatingS2}
S_2\lesssim
\sum_{n\geq 0}
\|f\|_{L^2(\mathcal{A}_{n})}^2
\Bigg(\sum_{m\geq n}2^{-\frac{m-n}2}
\min\{1,\eps \,2^{\frac{m-n}2}\}\Bigg)
\end{equation}
as before. 
In more detail, set $\kappa:=m-n$ and break up the sum in two pieces, depending on whether or not the condition $\eps \,2^{\frac {\kappa}2} < 1$ is satisfied. This yields:
$$\sum_{m\geq n}2^{-\frac{m-n}2}
\min\big\{1,\eps\, 2^{\frac{m-n}2}\big\}\leq
\sum_{\kappa< 2\log_2(\eps^{-1})} 2^{-\frac {\kappa}2}\,\eps\, 2^{\frac {\kappa}2}
+\sum_{\kappa\geq 2\log_2(\eps^{-1})} 2^{-\frac {\kappa}2}\lesssim\eps\log_2(\eps^{-1}).$$
Plugging this back into \eqref{EstimatingS2}, we finally obtain that $S_2\lesssim\eps\log_2(\eps^{-1})$.
This completes the proof.
\end{proof}

\begin{proposition}\label{SpecialCap2D}
Let $d=2$ and $4 \leq p < 6$. Let $\{f_n\}_{n\in\N}\subset L^2(\H^2)$ be an extremizing sequence for inequality \eqref{ExtensionInequality}, normalized so that $\|f_n\|_{L^2(\H^2)}=1$ for each $n \in \N$. There exists a universal constant $\eta_{2,p}>0$ and $n_0 \in \N$, such that for any $n \geq n_0$ there exist $s_n\in (-1,1)$ and $\varphi_n\in[0,2\pi)$ verifying
$$\| (L^{s_n}\circ R_{\varphi_n})^*f_n\|_{L^2(\D)}
\geq \eta_{2,p}\,,$$
where $\D:=\{(\xi,\tau)\in\H^2: |\xi|\leq 2\sqrt{2}\pi\}$.
\end{proposition}

\begin{proof}
Let $n_0 \in \N$ be such that, for $n \geq n_0$, we have 
$$\|T f_n\|_{L^p(\R^3)}\geq\frac{{\bf H}_{2,p}}2.$$
Fix $n\geq n_0$. We claim that there exists $\gamma_p>0$, depending only on $p$, such that
\begin{equation}\label{PreSpecialCap}
\sup_{m,\ell}\|f_n\|_{L^2(\C_{m,\ell})}\geq \gamma_p>0,
\end{equation}
where the supremum is taken over integers $m\geq 0$ and $0\leq\ell<2^m$.
For otherwise we could  appeal to Proposition \ref{SpecialCap2DMainProp} to ensure
$$\frac{{\bf H}_{2,p}}2\leq \|T f_n\|_{L^p(\R^3)}
\leq \|Tf_n\|_{L^4(\R^3)}^\theta \|Tf_n\|_{L^6(\R^3)}^{1-\theta}
\lesssim (\eps\log_2(\eps^{-1}))^{\frac{\theta}4},
$$
which is a contradiction provided $\theta>0$ and $\eps>0$ is sufficiently small.
Knowing \eqref{PreSpecialCap}, it is now a simple matter to invoke Lemma \ref{2Dtessellation} (b) and 
 conclude the proof of the proposition.
\end{proof}

\section{Concentration Compactness}\label{sec:CC}
In this section, we adapt parts of the work of Fanelli, Vega and Visciglia \cite{FVV11,FVV12} in order to
 complete the proof of Theorem \ref{Thm2}.
We rely on the following key result from \cite[Proposition 1.1]{FVV11}.
\begin{lemma} [cf. \cite{FVV11}]\label{FVVLemma}
  Let  $\mathcal{H}$ be a Hilbert space and $S:\mathcal{H}\to L^p(\R^d)$ be a bounded linear operator with $p\in(2,\infty)$.  Consider $\{h_n\}_{n\in\N}\subset \mathcal{H}$ such that
\begin{enumerate}
    \item[(i)] $\lim_{n\to\infty}\|h_n\|_{\mathcal{H}}=1$;
    \item[(ii)] $\lim_{n\to\infty}\|S(h_n)\|_{L^p(\R^d)}=\|S\|_{\mathcal{H}\to L^p(\R^d)}$;
    \item[(iii)] $h_n\rightharpoonup h\neq 0$;
    \item[(iv)] $S(h_n)\rightarrow S(h)$ almost everywhere in $\R^d$.
\end{enumerate}
Then $h_n\rightarrow h$ in $\mathcal{H}$. 
In particular, $\|h\|_{\mathcal{H}}=1$ and $\|S(h)\|_{L^p(\R^d)}=\|S\|_{\mathcal{H}\to L^p(\R^d)}$.
\end{lemma}
The argument which we will present next works as long as one can produce a special cap, as was done in \S\ref{sec:SpecialCap} in the lower dimensional cases $d\in\{1,2\}$.
We state the next two results in general dimensions $d$, thereby guaranteeing the existence of extremizers, conditionally on the existence of a special cap.

\begin{proposition}\label{NonZeroWeakLimit}
Let $d\geq 1$ and let $p$ be such that
  \begin{equation*}
   \begin{cases}
  6 < p< \infty, \text{ if } d=1;\\
  \frac{2(d+2)}{d} < p\leq  \frac{2(d+1)}{d-1}, \text{ if } d\geq 2.
      \end{cases} 
\end{equation*}
Assume the existence of two universal constants $\eta=\eta_{d,p}>0$ and $r = r_{d,p} >0$ verifying the following property: for any extremizing sequence $\{f_n\}_{n\in\N}\subset L^2(\H^d)$ for inequality \eqref{ExtensionInequality}, normalized so that $\|f_n\|_{L^2(\H^d)}=1$ for each $n \in \N$, there exists $n_0 \in \N$ such that 
$$\|f_n\|_{L^2(\D)}\geq \eta$$
for any $n \geq n_0$, where $\D:=\{(\xi,\tau)\in\H^d: |\xi|\leq r\}$. Then there exists $(x_n,t_n)\in \R^d\times\R$ such that the sequence $\{g_n\}_{n\in\N}$ defined by
$$g_n(y):=e^{i x_n\cdot y}e^{it_n\langle y\rangle}f_n(y)$$
admits a subsequence that converges weakly to a nonzero limit in $L^2(\H^d)$.
\end{proposition}

\begin{proof}
We follow the outline of the proof of \cite[Theorem 1.1, $p\neq \infty$]{FVV11}. Setting $f_{n,0}:=f_n\one_{\D}$ we have, for $n \geq n_0$,
\begin{equation}\label{OrthogonalDec}
\eta\leq \|f_{n,0}\|_{L^2(\H^d)}\leq 1,
\ \text{ and }\ \|f_n\one_{\H^d\setminus \D}\|_{L^2(\H^d)}\leq (1-\eta^2)^{\frac 12}.
\end{equation}
Moreover, 
$$T(f_{n,0})(x,t)=\int_{\{|y|\leq r\}} e^{ix\cdot y}e^{it\langle y\rangle}f_n(y)\frac{\d y}{\langle y\rangle}$$
is a smooth function of $x,t$, satisfying 
\begin{align}
\|T(f_{n,0})\|_{L^\infty(\R^{d+1})}&\lesssim \|f_{n,0}\|_{L^2(\H^d)}\leq1,\label{LinftyUpperBound}\\
\|\nabla_{x,t} T(f_{n,0})\|_{L^\infty(\R^{d+1})}&\lesssim\|f_{n,0}\|_{L^2(\H^d)}\leq1.\label{LinftyGradUpperBound}
\end{align}

Since $\frac{2(d+2)}{d} < p$, the log-convexity of Lebesgue norms, together with the sharp inequality \eqref{ExtensionInequality} and the first upper bound in \eqref{OrthogonalDec}, yields
\begin{equation}\label{FirstLowerBound}
\|T(f_{n,0})\|_{L^p(\R^{d+1})}
\leq \|T(f_{n,0})\|_{L^{\frac{2(d+2)}d}(\R^{d+1})}^{\frac{2(d+2)}{dp}}
\|T(f_{n,0})\|_{L^\infty(\R^{d+1})}^{\frac{d(p-2)-4}{dp}}
\leq{\bf H}_{d,\frac{2(d+2)}{d}}^{\frac{2(d+2)}{dp}}\|T(f_{n,0})\|_{L^\infty(\R^{d+1})}^{\frac{d(p-2)-4}{dp}}.
\end{equation}
Since the sequence $\{f_n\}_{n\in\N}$ is extremizing and $L^2$-normalized, there exists $\delta=\delta_{d,p}>0$, depending only on $d$ and $p$, for which
$$\|Tf_n\|_{L^p(\R^{d+1})}\geq \big(\delta+(1-\eta^2)^{\frac12} \big)\,{\bf H}_{d,p}\,,$$
for every sufficiently large $n\in\N$. 
Together with the second upper bound in \eqref{OrthogonalDec}, this implies
\begin{align}
\|T(f_{n,0})\|_{L^p(\R^{d+1})}
&\geq \|Tf_n\|_{L^p(\R^{d+1})}-\|T(f_n\one_{\H^d\setminus \D})\|_{L^p(\R^{d+1})}\notag\\
&\geq \big(\delta+(1-\eta^2)^{\frac12} \big)\,{\bf H}_{d,p}-(1-\eta^2)^{\frac12}\,{\bf H}_{d,p}
=\delta\,{\bf H}_{d,p}.\label{SecondLowerBound}
\end{align}
From \eqref{FirstLowerBound} and \eqref{SecondLowerBound} we get
$$\|T(f_{n,0})\|_{L^\infty(\R^{d+1})}\geq \gamma:=\delta^{\frac {dp}{d(p-2)-4}}\,{\bf H}_{d,p}^{\frac {dp}{d(p-2)-4}}\,{\bf H}_{d,\frac{2(d+2)}{d}}^{-\frac{2(d+2)}{d(p-2)-4}}.$$
This readily implies the existence of $(x_n,t_n)\in\R^d\times\R$, for which
\begin{equation}\label{LinftyLowerBound}
|T(f_{n,0})(x_n,t_n)|\geq\frac{\gamma}{2}.
\end{equation}
Setting 
\begin{equation}\label{SeqTildeG}
\widetilde{g}_n(y):=e^{ix_n\cdot y}e^{it_n\langle y\rangle}f_{n,0}(y),
\end{equation}
we have that $\|\widetilde{g}_n\|_{L^2(\H^d)}=\|f_{n,0}\|_{L^2(\H^d)}$. 
Moreover, $T(\widetilde{g}_n)$ amounts to a space-time translation of the function $T(f_{n,0})$.
From \eqref{LinftyUpperBound}, \eqref{LinftyGradUpperBound} and \eqref{LinftyLowerBound}, it then follows that  
\begin{equation}\label{LowerUpperBoundsTg}
\|T(\widetilde{g}_n)\|_{L^\infty(\R^{d+1})}\lesssim 1,\;
\|\nabla_{x,t} T(\widetilde{g}_n)\|_{L^\infty(\R^{d+1})}\lesssim 1,\text{ and }
|T(\widetilde{g}_n)(0,0)|\geq\frac{\gamma}2.
\end{equation}
The implicit constants in the first and second estimates in \eqref{LowerUpperBoundsTg} are independent of $n$, and so  the sequence $\{T(\widetilde{g}_n)\}_{n\in\N}$ is uniformly bounded and equicontinuous on the unit cube $[-\frac12,\frac12]^{d+1}$.
The Arzel\`a--Ascoli Theorem on $\R^{d+1}$ then implies that the sequence $\{T(\widetilde{g}_n)\}_{n\in\N}$ has a subsequence which converges uniformly to a  limit. That this limit is nonzero follows at once from the third estimate in \eqref{LowerUpperBoundsTg}.

Now, since the sequence $\{\widetilde{g}_n\}_{n\in\N}$ is  bounded on $L^2(\H^d)$, it has a weakly convergent subsequence. In other words, we may thus assume, possibly after extraction, that there exists a function $\widetilde{g}\in L^2(\H^d)$, such that $\widetilde{g}_n\rightharpoonup \widetilde{g}$ weakly in $L^2(\H^d)$, as $n\to\infty$.
Since the operator $T$ is bounded from $L^2(\H^d)$ to $L^p(\R^{d+1})$, it follows that $T(\widetilde{g}_n)\rightharpoonup T(\widetilde{g})$ weakly in $L^p(\R^{d+1})$, as $n\to\infty$.
From the previous paragraph  we conclude that $T(\widetilde{g})$ is nonzero, and so the function $\widetilde{g}$ is itself nonzero.

This implies that the sequence $\{g_n\}_{n\in\N}$ defined by
$${g}_n(y):=e^{ix_n\cdot y}e^{it_n\langle y\rangle}f_{n}(y),$$
where the parameters $(x_n,t_n)$ are those from \eqref{SeqTildeG},
has a subsequence which converges weakly to a nonzero limit.
Indeed, if $g\in L^2(\H^d)$ is such that $g_n\rightharpoonup g$ weakly in $L^2(\H^d)$, as $n\to\infty$, then $g_n\one_\D\rightharpoonup g\one_\D$ weakly in $L^2(\H^d)$, as $n\to\infty$. 
Therefore, in order to prove that $g$ is nonzero, it suffices to show that it has nonzero mass inside $\D$. This follows from the fact that $\widetilde{g}$ is nonzero, which we checked in the last paragraph.
The proof of the proposition is now complete.  
\end{proof}

\begin{proposition}\label{AeConvergenceT}
Let $d\in\N$, and let $p$ be such that
  \begin{equation*}
   \begin{cases}
  6 < p< \infty, \text{ if } d=1;\\
\frac{2(d+2)}{d} < p\leq  \frac{2(d+1)}{d-1}, \text{ if } d\geq 2.
      \end{cases}  
\end{equation*}
Let $\{f_n\}_{n\in\N}\subset L^2(\H^d)$ be an extremizing sequence for inequality \eqref{ExtensionInequality}, normalized so that $\|f_n\|_{L^2(\H^d)}=1$ for each $n \in\N$, which converges weakly to a nonzero limit $f\in L^2(\H^d)$. Then, possibly after passing to a subsequence,
$$Tf_n(x,t)\to Tf(x,t), \text{ as } n\to\infty,$$
for almost every $(x,t)\in\R^d\times\R$.
\end{proposition}

\begin{proof}
We follow the outline of the proof of \cite[Theorem 1.1]{FVV12}. For each $n\in\N$, define the auxiliary functions
$$\widehat{g}_n(y):=\frac{f_n(y)}{\langle y\rangle},\text{ and also }
\widehat{g}(y):=\frac{f(y)}{\langle y\rangle}.$$
As it has been pointed out in \eqref{July19_12:43am}, the extension operator on the hyperboloid and the Klein--Gordon propagator are related by
$$Tf_n(x,t)=(2\pi)^{d} \,e^{it\sqrt{1-\Delta}} g_n(x),$$
and it suffices to show that,  pointwise for almost every  $(x,t)\in\R^d\times\R$,
$$e^{it\sqrt{1-\Delta}} g_n(x)\to e^{it\sqrt{1-\Delta}} g(x), \text{ as }n\to\infty,$$
possibly after extraction of a subsequence. For $t\in\R$ and $R>0$, we define
$$F_n(t,R):=\int_{\{|x|\leq R\}} \big|e^{it\sqrt{1-\Delta}}(g_n-g)(x)\big|^2 \d x,$$
and we claim that
\begin{equation}\label{RellichClaim}
\lim_{n\to\infty} F_{n}(t,R)=0
\end{equation}

In order to prove this claim, first recall that $\{g_n\}_{n\in\N}$ is bounded on the Sobolev space $H^{\frac12}(\R^d)$, with
\begin{equation}\label{BoundednessH12}
\|g_n\|_{H^{\frac12}(\R^d)}=\|f_n\|_{L^2(\H^d)}=1.
\end{equation}
Let $B_R\subset\R^d$ denote the ball centered at the origin of radius $R$. The hypothesis $f_n\rightharpoonup f$ weakly in $L^2(\H^d)$ can be equivalently restated as $g_n\rightharpoonup g$ weakly in $H^{\frac12}(\R^d)$. Since, for fixed $t\in\R$, the operator $e^{it\sqrt{1-\Delta}}$ is unitary on $H^{\frac12}(\R^d)$, it follows that $e^{it\sqrt{1-\Delta}} g_n\rightharpoonup e^{it\sqrt{1-\Delta}} g$ weakly in $H^{\frac12}(\R^d)$, which in turn implies that $e^{it\sqrt{1-\Delta}} g_n\rightharpoonup e^{it\sqrt{1-\Delta}} g$ weakly in $H^{\frac12}(B_R)$. As a consequence of \eqref{BoundednessH12} and of Rellich's Theorem, see e.g. \cite[Theorem 7.1 and Proposition 3.4]{DPV12}, we find
$$e^{it\sqrt{1-\Delta}} g_{n}\to e^{it\sqrt{1-\Delta}} g \text{ strongly in }L^2(B_R), \text{ as }n\to\infty.$$
In other words, \eqref{RellichClaim} holds as claimed. We further note, since the operator $e^{it\sqrt{1-\Delta}}$ is unitary on $L^2(\R^d)$, that
$$F_n(t,R)\leq \big\|e^{it\sqrt{1-\Delta}}(g_n-g)\big\|_{L^2(\R^d)}^2=\|g_n-g\|_{L^2(\R^d)}^2\ \lesssim\ \|g_n-g\|_{H^{\frac12}(\R^d)}^2\lesssim 1.$$

This justifies the applicability of  Lebesgue's Dominated Convergence Theorem which, together with \eqref{RellichClaim}, implies
$$\lim_{n\to\infty}\int_{-R}^R F_n(t,R) \,\d t =
\lim_{n\to\infty}\int_{-R}^R  \int_{\{|x|\leq R\}} \big|e^{it\sqrt{1-\Delta}}(g_n-g)(x)\big|^2 \,\d x\,\d t
=0.$$
As a consequence, up to a subsequence,
$$e^{it\sqrt{1-\Delta}}(g_n-g)(x)\to 0, \text{ as }n\to\infty, \text{ for a.e. } (x,t)\in B_R\times(-R,R).$$
The extraction of the subsequence depends on the radius $R$. 
To remedy this, repeat the argument on a discrete sequence of radii $\{R_j\}_{j\in\N}$ satisfying $R_j\to\infty$, as $j\to\infty$, to conclude, via a standard diagonal argument, that there exists a subsequence $\{g_{n_k}\}_{k\in\N}\subset\{g_n\}_{n\in\N}$ such that
$$e^{it\sqrt{1-\Delta}}(g_{n_k}-g)(x)\to 0, \text{ as }k\to\infty, \text{ for a.e. } (x,t)\in \R^d\times\R.$$
This concludes the proof of the proposition.
\end{proof}

It is now an easy matter to finish the proof of Theorem \ref{Thm2}.
\begin{proof}[Proof of Theorem \ref{Thm2}]
Let us start by considering the case $d=1$ and $p\in(6,\infty)$.
The strategy is to invoke Lemma \ref{FVVLemma} with $S=T$ and $\mathcal{H}=L^2(\H^1)$.
With that purpose in mind, let $\{f_n\}_{n\in\N}$ be an extremizing sequence for the inequality
\begin{equation}\label{Sharp1DGenp}
\|Tf\|_{L^p(\R^{2})}\leq {\bf H}_{1,p} \|f\|_{L^2(\H^1)},
\end{equation}
normalized so that $\|f_n\|_{L^2(\H^1)}=1$ for each $n \in \N$. In particular, conditions (i) and (ii) from Lemma \ref{FVVLemma} are automatically met.
We will be done once we check that conditions (iii) and (iv) hold as well. By Proposition \ref{SpecialCap1D}, the sequence $\{(L^{s_n})^*f_n\}_{n\in\N}$, which is still extremizing for \eqref{Sharp1DGenp}, verifies
$$\|(L^{s_n})^*f_n\|_{L^2(\C_0)}\geq \eta_{1,p}>0,$$
for every $n\in\N$.
By Proposition \ref{NonZeroWeakLimit}, the sequence $\{g_n\}_{n\in\N}$ defined by 
$$g_n(y)=e^{ix_n y}e^{it_n\langle y\rangle}((L^{s_n})^*f_n)(y),$$
which is still extremizing for \eqref{Sharp1DGenp}, is such that 
$$g_n\rightharpoonup g\neq 0 \text{ weakly in }L^2(\H^1),\text{ as }n\to\infty,$$
possibly after passing to a subsequence. By Proposition \ref{AeConvergenceT}, we then know that
$$T(g_n)\to T(g)\text{ pointwise a.e. on $\R^2$, as }n\to\infty,$$
again possibly after passing to a subsequence. By Lemma \ref{FVVLemma}, we finally conclude that $g_n\to g$ in $L^2(\H^1)$, as $n\to\infty$.
In other words, $g$ is an extremizer for inequality \eqref{Sharp1DGenp}.
This concludes the proof of the one-dimensional case.
The two-dimensional case $d=2$ and $p\in (4,6)$ can be handled in an analogous way. One just invokes
Proposition \ref{SpecialCap2D} instead of Proposition \ref{SpecialCap1D}, the rest of the argument being identical.
This concludes the proof.
\end{proof}

\section*{Acknowledgements}

We thank Ren\'e Quilodr\'an, Betsy Stovall and Christoph Thiele for stimulating discussions. This work was accomplished during visits to the Hausdorff Research Institute for Mathematics (Bonn), the International Centre for Theoretical Physics (Trieste) and Stanford University, whose hospitality is greatly appreciated. E.C. acknowledges support from CNPq-Brazil, FAPERJ-Brazil and the Fulbright Junior Faculty Award. D.O.S. was partially supported by the Hausdorff Center for Mathematics and DFG grant CRC 1060. M.S. acknowledges support from FAPERJ-Brazil.

\end{document}